\documentclass[11pt,reqno]{article}
 
\usepackage[margin=1in]{geometry} 
\usepackage{amsmath,amsthm,amssymb,bbm,amsfonts,color,url,enumitem,hyperref}
\usepackage{graphicx}
\usepackage{mathtools}
\usepackage{hyperref}

\usepackage[margin=2.5cm]{caption}

\newcommand{\R}{\mathbb{R}}

\newcommand{\Fc}{\mathcal{F}}
\newcommand{\I}{\mathbbm{1}}
\newcommand{\Nc}{\mathcal{N}}
\newcommand{\Var}{\mathrm{Var}}

\newcommand{\PF}{\mathrm{PF}}
\newcommand{\PPF}{\mathrm{PPF}}
\newcommand{\E}{\mathbb{E}}
\newcommand{\PR}{\mathrm{P}}
\newcommand{\VarR}{\mathrm{Var}}
\newcommand{\Ber}{\mathtt{Bernoulli}}
\newcommand{\Rayleigh}{\mathtt{Rayleigh}}

\def\bal#1\eal{\begin{align*}#1\end{align*}}
\def\bal#1\eal{\begin{align*}#1\end{align*}}

\newtheorem{theorem}{Theorem}[section]
\newtheorem{lemma}[theorem]{Lemma}
\newtheorem{proposition}[theorem]{Proposition}

\newtheorem{remark}[theorem]{Remark}

\newcommand{\Address}{
\bigskip
\footnotesize

\textsc{Department of Mathematics \& Statistics, McMaster University, Hamilton, ON, Canada} \par\nopagebreak
\textit{E-mail address}: paguyoj@mcmaster.ca \\

\textsc{Department of Mathematics, University of Denver, Denver, CO, USA} \par\nopagebreak
\textit{E-mail address}: mei.yin@du.edu}

\begin{document}
 
\title{Limit distributions for cycles of random parking functions} 
\author{J. E. Paguyo and Mei Yin\thanks{M.~Yin was supported in part by the Simons Foundation Grant MPS-TSM-00007227.}} 
\date{}
\maketitle

\abstract{
We study the asymptotic behavior of cycles of uniformly random parking functions. Our results are multifold: we obtain an explicit formula for the number of parking functions with a prescribed number of cyclic points and show that the scaled number of cyclic points of a random parking function is asymptotically Rayleigh distributed; we establish the classical trio of limit theorems (law of large numbers, central limit theorem, large deviation principle) for the number of cycles in a random parking function; we also compute the asymptotic mean of the length of the $r$th longest cycle in a random parking function for all valid $r$. A variety of tools from probability theory and combinatorics are used in our investigation. Corresponding results for the class of prime parking functions are obtained.
}

\section{Introduction}

Parking functions were introduced by Konheim and Weiss \cite{KW66} in their study of the hash storage structure and have since found wide applications to combinatorics, probability, and computer science. We refer the reader to Yan \cite{Yan15} for an accessible and extensive survey of combinatorial results. 

Consider $n$ parking spots placed sequentially on a one-way street. A line of $n$ cars, labeled $1 \leq i \leq n$, enters the street one at a time. The $i$th car drives to its preferred parking spot $\pi_n(i)$ and parks if the spot is available. Otherwise, if the spot is occupied, it parks in the first available spot after $\pi_n(i)$. If a car is unable to find any available spots, then it exits the street without parking. A sequence of parking preferences $\pi_n = (\pi_n(1), \ldots, \pi_n(n))$ is a {\em (classical) parking function} of length $n$ if all $n$ cars are able to park. 

Let $[n] := \{1,\ldots, n\}$. By the pigeonhole principle, a sequence $\pi_n = (\pi_n(1), \ldots, \pi_n(n)) \in [n]^n$ is a parking function of length $n$ if and only if $\pi_{n,(i)} \leq i$ for all $i \in [n]$, where $\pi_{n,(1)} \leq \dotsb \leq \pi_{n,(n)}$ is the weakly increasing rearrangement of $\pi_n$. Equivalently, $\pi_n$ is a parking function if and only if $|\{k : \pi_n(k) \leq i\}| \geq i$ for all $i \in [n]$. This implies that parking functions are invariant under permutations of the coordinates. 

Let $\PF_n$ be the set of parking functions of length $n$. The number of parking functions is
    $|\PF_n| = (n+1)^{n-1}$.
An elegant, unpublished proof of this result using a circular symmetry argument was given by Pollak and recounted in \cite{FR74} and \cite{Pollak2}. A parking function $\pi_n = (\pi_n(1), \ldots, \pi_n(n))$ is a {\em prime parking function} if for all $1 \leq j \leq n-1$, at least $j+1$ cars have a parking preference in the first $j$ parking spots. Equivalently, $\pi_n$ is a prime parking function if it remains a parking function even after removing a coordinate that equals $1$. Let $\PPF_n$ be the set of prime parking functions of length $n$. Kalikow \cite{Kal99} modified Pollak's circular argument to show that the number of prime parking functions is
    $|\PPF_n| = (n-1)^{n-1}$.

The probabilistic study of parking functions is a more recent line of work \cite{Bel23, CM01, DH17, DHHRY23, FPV98, Har25, Jan01, KY21, Pag23, SY23, Yin23I, Yin23II}. 
Of particular importance to us is the probabilistic program, initiated by Diaconis and Hicks in \cite{DH17}, of studying the distribution of statistics of uniformly random parking functions.

\subsection{Literature review}

A {\em mapping} of length $n$ is a function $f_n: [n] \to [n]$. Let $\Fc_n$ be the set of mappings of length $n$, so that $|\Fc_n| = n^n$. Observe that a parking function is also a mapping, so that $\PF_n \subseteq \Fc_n$. At the other end of the spectrum, let $S_n$ be the set of {\em permutations} of $[n]$. It is clear that any permutation is also a parking function, so that $S_n \subseteq \PF_n$.

For any mapping $f_n = (f_n(1), \ldots, f_n(n)) \in \Fc_n$ of length $n$, we can define its {\em digraph} or {\em digraph representation}, $G_{f_n}$, as the directed graph with vertex set $[n]$ and a directed edge from $i$ to $f_n(i)$ for all $i \in [n]$. Observe that every vertex has outdegree one and that fixed points are represented by a vertex with a self-loop. Therefore digraphs of mappings, and hence parking functions, and hence permutations, consist of a collection of connected {\em components}, with each component consisting of a {\em cycle} along with {\em tree components} attached to vertices of the cycle. Note that all tree components are trivial in the digraphs of permutations, containing a root only with no hanging branches. See Figure \ref{fig:figure1} for the digraph representation of the parking function $\pi_{20}=(18, 4, 2, 19, 12, 2, 2, 6, 10, 13, 5, 5, 18, 15, 9, 1, 17, 3, 10, 4) \in \PF_{20}$.

\begin{figure}[h!]
\centering
\includegraphics[scale = 0.4]{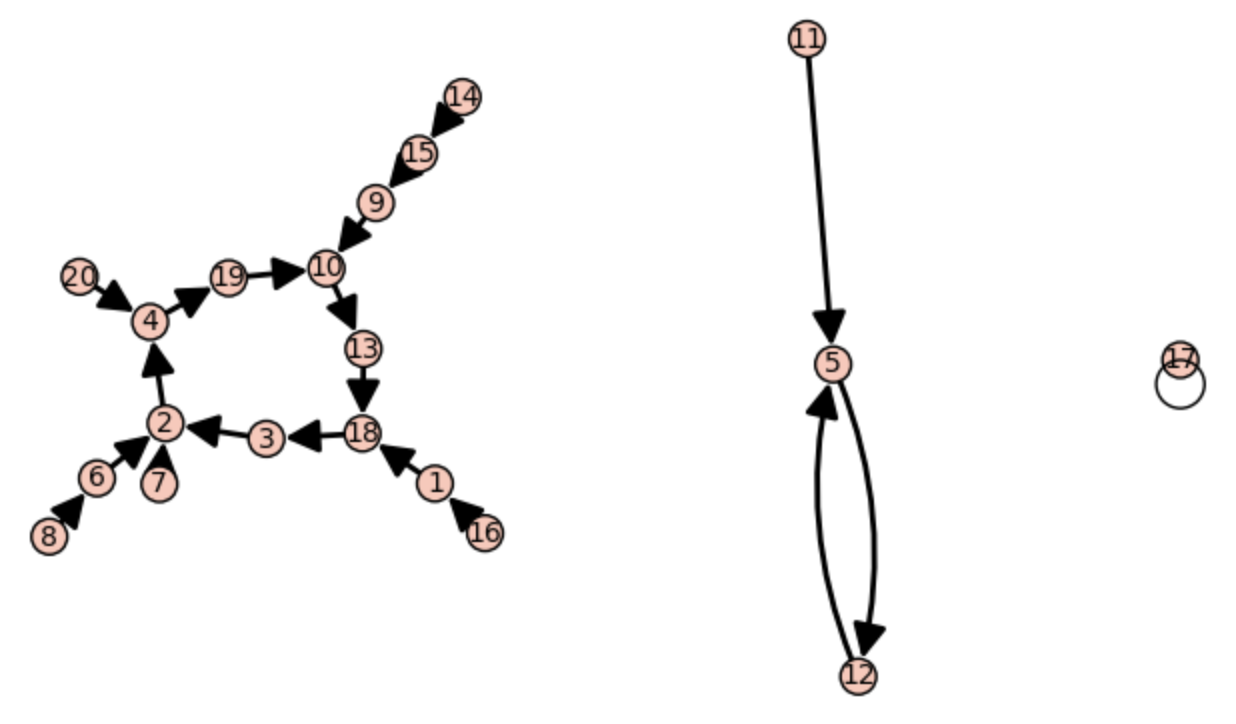}
\caption{The digraph representation of the parking function \\ $\pi_{20}=(18, 4, 2, 19, 12, 2, 2, 6, 10, 13, 5, 5, 18, 15, 9, 1, 17, 3, 10, 4)$ of length 20.}
\label{fig:figure1}
\end{figure}

Various statistics on mappings can be defined in terms of statistics on its corresponding digraph. Let $f_n \in \Fc_n$ be a random mapping, so that its corresponding digraph $G_{f_n}$ is a random graph. Let $K_n(f_n)$ be the {\em number of cycles} of $f_n$, defined as the number of cycles in the digraph $G_{f_n}$. Since each component in $G_{f_n}$ contains exactly one cycle (fixed points are special cases of cycles), $K_n(f_n)$ also equals the {\em number of components} of $f_n$. For all $k \in [n]$, let $C_k(f_n)$ be the {\em number of $k$-cycles} of $f_n$, defined as the number of cycles of length $k$ in $G_{f_n}$. A {\em cyclic point} of $f_n$ is a vertex which lies on a cycle in $G_{f_n}$. Let $\lambda_n(f_n)$ be the {\em number of cyclic points} of $f_n$. Finally, let $L_r(f_n)$ be the {\em length of the $r$th longest cycle} of $f_n$, defined as the length of the $r$th longest cycle in $G_{f_n}$.

The study of the cycle structure of random permutations has a rich history. We only survey some of these results here and refer the reader to Arratia, Barbour, and Tavar\'{e} \cite{ABT} for a treatise on random permutations and logarithmic combinatorial structures. The central limit theorem for the number of cycles in a random permutation $\sigma_n \in S_n$ was first established by Goncharov \cite{Gon44}, who showed that $K_n(\sigma_n)$ is asymptotically normal with mean and variance $\log n$. DeLaurentis and Pittel \cite{DP85} refined this by establishing a functional central limit theorem. Goncharov \cite{Gon44} also studied the joint distribution of small cycle counts $(C_1(\sigma_n), C_2(\sigma_n), \ldots)$ and showed that it converges in distribution to a process of independent Poisson random variables $(Z_1, Z_2, \ldots)$, where $Z_k$ is Poisson distributed with rate $1/k$ for all $k$. Arratia and Tavar\'{e} \cite{AT92} extended this result by obtaining a convergence rate for the distributional approximation. Shepp and Lloyd \cite{SL66} studied the length of the $r$th longest cycle, $L_r(\sigma_n)$, and obtained an expression for the asymptotic mean. Subsequently, Kingman \cite{Kin77} proved that the process of normalized longest cycle lengths $\left(\frac{L_1(\sigma_n)}{n}, \frac{L_2(\sigma_n)}{n}, \ldots \right)$ converges to the Poisson-Dirichlet distribution with parameter $1$. The large deviation principle for the number of cycles was established by Feng and Hoppe \cite{FH98}, where a more general large deviation principle was obtained for the Ewens-Pitman sampling model. 

The study of the cycle and component structure of random mappings was initiated by Harris \cite{Har60} who obtained formulas for distributions of several statistics. The central limit theorem for the number of cycles (equivalently, components) in a random mapping $f_n \in \Fc_n$ was established by Stepanov \cite{Ste69}. He showed that the number of cyclic points $\lambda_n(f_n)$, scaled by $\sqrt{n}$, is asymptotically $\Rayleigh(1)$ distributed, and used this to show that $K_n(f_n)$ is asymptotically normal with mean and variance $\frac{1}{2} \log n$. Thus mappings have on average half as many cycles as permutations. Hansen \cite{Han89} refined this result by establishing the functional central limit theorem for $K_n(f_n)$. A large deviation principle for the number of components was obtained by Hwang \cite{Hwa96}, where a more general result for combinatorial structures was established. Purdom and Williams \cite{PW68} studied the length of the $r$th longest cycle, $L_r(f_n)$, and computed the asymptotic mean. The analogous result for the size of the $r$th largest component was obtained by Kolchin \cite{Kol76}. The result that the process of normalized longest cycle lengths $\left(\frac{L_1(f_n)}{\lambda_n(f_n)}, \frac{L_2(f_n)}{\lambda_n(f_n)}, \ldots \right)$ converges in distribution to the Poisson-Dirichlet distribution with parameter $1$ seems to be folklore according to \cite{ABT}. On the other hand, Aldous \cite{Ald85} showed that the process of normalized largest component sizes $\left(\frac{\tilde{L}_1(f_n)}{n}, \frac{\tilde{L}_2(f_n)}{n}, \ldots \right)$ converges to the Poisson-Dirichlet distribution with parameter $\frac{1}{2}$, where $\tilde{L}_r(f_n)$ is the size of the $r$th largest component. The small cycle counts $C_k(f_n)$ were investigated by Flajolet and Odlyzko \cite{FO90} where they established Poisson limit theorems. In \cite{AP94}, Aldous and Pitman showed that the asymptotic distribution of various functionals of uniformly random mappings are functionals of Brownian bridge (see also \cite{Pit01} and Chapter 9 of \cite{Pit06}). For example, the asymptotic distribution of the diameter of the digraph of a random mapping is a functional of a reflecting Brownian bridge \cite{AP02}. More recently, Mutafchiev and Finch \cite{MF24} studied the {\em deepest cycle} of a random mapping, which is the cycle contained in the largest component, and showed that the length of the deepest cycle scaled by $\sqrt{n}$ converges in distribution to $\sqrt{\chi^2(1)\mu}$, where $\chi^2(1)$ is the standard chi-squared random variable and $\mu$ is the limiting distribution of the size of the largest component scaled by $n$. 

A unifying theme in the probabilistic exploration into parking functions is the notion of the {\em equivalence of ensembles}.
Diaconis and Hicks \cite{DH17} stated that for certain statistics, it is natural to expect the distribution of statistics in the ``micro-canonical ensemble", $\PF_n$, to be close to the distribution of statistics in the ``canonical ensemble", $\Fc_n$. Indeed, they showed that the equivalence of ensembles holds for statistics such as the number of repeats, lucky cars, and descents. However, they also showed that it fails for some statistics such as the distribution of the value of the first coordinate.

As an open problem, Diaconis and Hicks suggested the study of the cycle structure of a random parking function. Not much progress has been made as of yet.
Paguyo \cite{Pag23} used Stein's method to show for a uniformly random parking function, $\pi_n \in \PF_n$, the process of small cycle counts $(C_1(\pi_n), C_2(\pi_n), \ldots)$ converges in distribution to a process of independent Poisson random variables $(Z_1, Z_2, \ldots)$, where $Z_k$ is Poisson distributed with rate $1/k$ for all $k$. This shows that the asymptotic behavior of the small cycle counts of random parking functions coincides with that of random permutations and random mappings. 
Rubey and Yin \cite{RY25} obtained exact formulas for the number of parking functions with exactly $k$ many $m$-cycles and for the expected number of $m$-cycles. 

\subsection{Our contributions}
\begin{figure}[h!]
\centering
    \includegraphics[width=.5\textwidth]{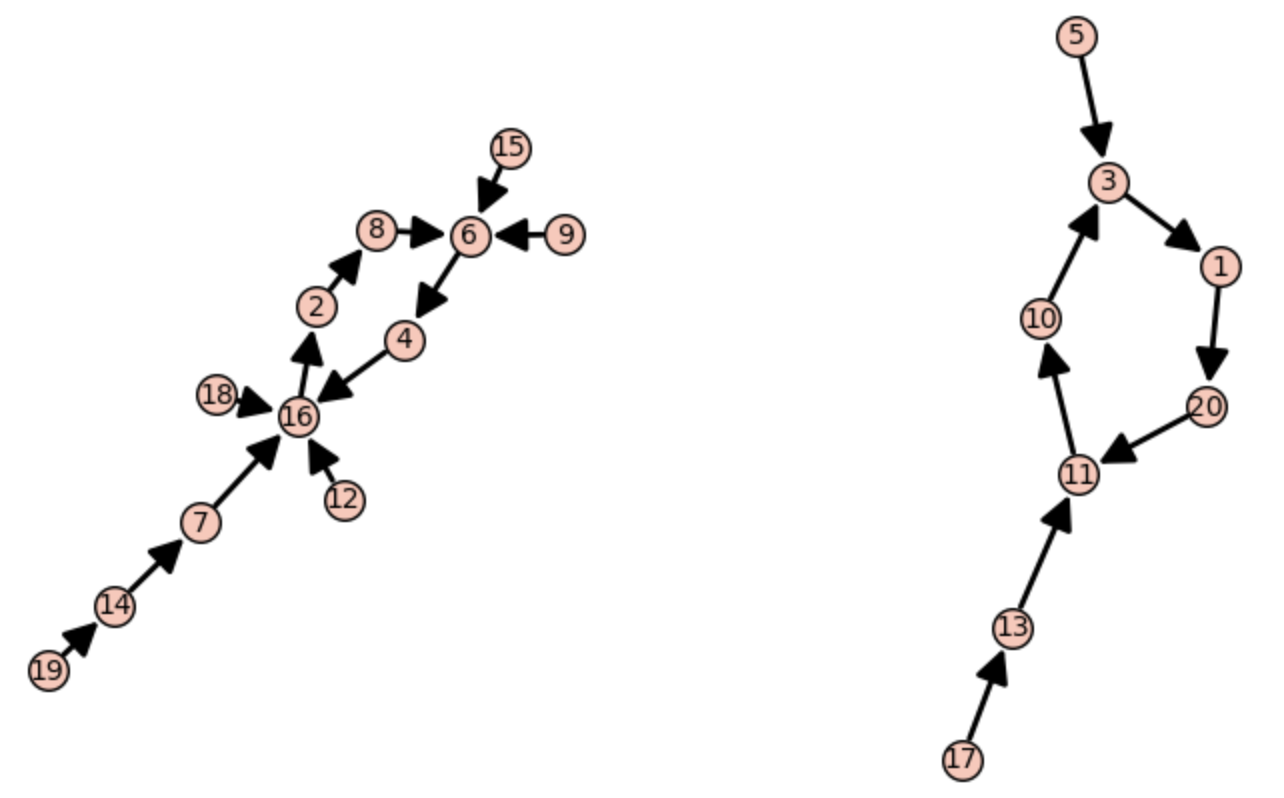}
\hfill
    \includegraphics[width=.5\textwidth]{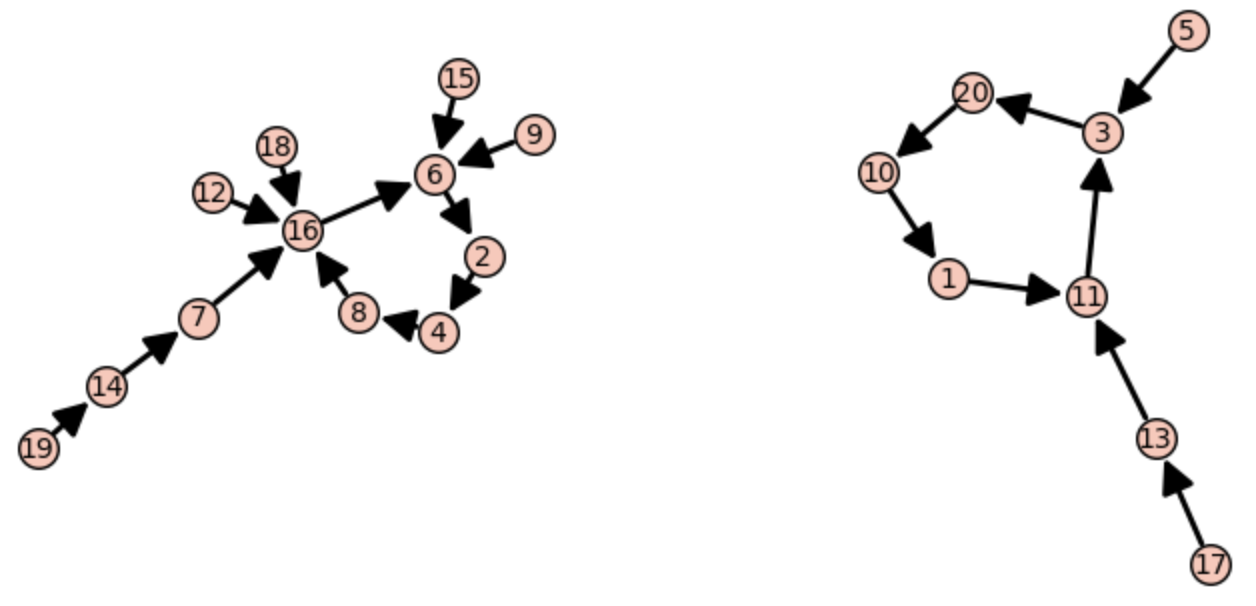}
\hfill
    \includegraphics[width=.5\textwidth]{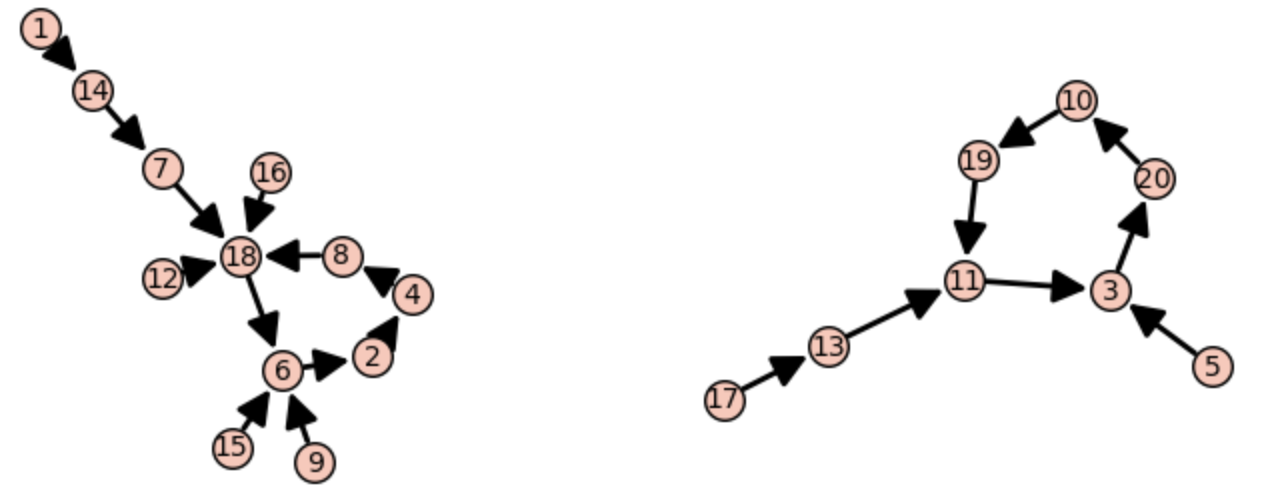}
    \caption{The digraph representations of parking functions $\pi^1_{20}$ (top plot) and $\pi^2_{20}$ (middle plot) and non-parking function $\pi^3_{20}$ (bottom plot). Both $\pi^2_{20}$ and $\pi^3_{20}$ are transformed from $\pi^1_{20}$ by permuting vertices.}
    \label{fig:figure2}
\end{figure}

Understanding the cycle structure of random parking functions has been a challenging research topic, as random parking functions do not satisfy nice \emph{exchangeability} properties like random permutations and random mappings do. Permuting the vertices in the digraph representation of a permutation (resp. mapping) in any fashion always yields a permutation (resp. mapping), but permuting the vertices in the digraph representation of a parking function does not necessarily yield another parking function. Still, there is a glimpse of hope. In each tree of the digraph representation of a parking function, there is a singled out root in a natural way, namely, the cyclic vertex to which the tree is attached. Permuting the rooted trees (each treated as an integral part) along the cycle where the root belongs in the digraph representation of a parking function in any fashion always yields a parking function. The investigations in this paper will rely heavily on the fact that the cyclic vertices of $\pi_n$ form a uniform permutation of size $\lambda_n(\pi_n)$, cf. Remark \ref{rk:exp_cyclic}.
See the top plot in Figure \ref{fig:figure2} for the digraph representation of a parking function $\pi^1_{20}=(20, 8, 1, 16, 3, 4, 16, 6, 6, 3, 10, 16, 11, 7, 6, 2, 13, 16, 14, 11) \in \PF_{20}$. Permuting the tree components attached to the cyclic points maintains the parking condition, yielding $\pi^2_{20}=(11, 4, 20, 8, 3, 2, 16, 16, 6, 1, 3, 16, 11, 7, 6, 6, 13, 16, 14, 10)$, which is a valid parking function. Permuting other vertices nevertheless might violate the parking condition, yielding $\pi^3_{20}=(14, 4, 20, 8, 3, 2, 18, 18, 6, 19, 3, 18, 11, 7, 6, 18, 13, 6, 11, 10)$, which is not a valid parking function. See the middle and bottom plots in Figure \ref{fig:figure2}.

In this paper, we demonstrate the asymptotic equivalence of ensembles between parking functions and mappings concerning the cyclic points via an in-depth exploration of the cycle structure of random parking functions. A variety of tools from probability theory and combinatorics are used. Let $\pi_n \in \PF_n$ be a uniformly random parking function. Our contributions are:

\begin{itemize}
\item Theorems \ref{classical-cyclic} and \ref{prime-cyclic}: Formulas for the number of classical and prime parking functions with exactly $k$ cyclic points. 
\item Proposition \ref{PFcyclicPDF}: The scaled number of cyclic points, $\frac{\lambda_n(\pi_n)}{\sqrt{n}}$, is asymptotically $\Rayleigh(1)$ distributed. 
\item Theorem \ref{PFcyclesLLN} (Law of large numbers): The scaled number of cycles, $\frac{K_n(\pi_n)}{\frac{1}{2} \log n}$, converges in probability to $1$. 
\item Theorem \ref{PFcyclesCLT} (Central limit theorem): The number of cycles, $K_n(\pi_n)$, is asymptotically normal with mean and variance $\frac{1}{2}\log n$. 
\item Theorem \ref{PFcyclesLDP} (Large deviation principle): The family $\left\{\frac{K_n(\pi_n)}{\log n} \right\}_{n \geq 1}$ satisfies a large deviation principle with speed $\log n$ and an explicit rate function $I(x)$. 
\item Theorem \ref{LongCyclesAsymptoticMean}: The asymptotic mean of the scaled length of the $r$th longest cycle, $\frac{L_r(\pi_n)}{\sqrt{n}}$, for all valid $r$, for classical parking functions.
\item Theorem \ref{PPFLimitTheorems}: The corresponding limit theorems for the class of prime parking functions. 
\end{itemize}

Our results show that uniformly random mappings and uniformly random parking functions behave similarly not just asymptotically, but also combinatorially. Earlier in the introduction, we pointed out that Pollak's circle argument on mappings from $[n]$ to $[n+1]$ reveals hidden connections between parking functions and mappings. We will use an adaptation of Pollak's argument in our proof of Theorem \ref{classical-cyclic}, which shows that the number of parking functions of length $n$ with $k$ cyclic points is $k k! (n+1)^{n-k-2} \binom{n+1}{k}$ for $1\leq k<n$ and $n!$ for $k=n$. In comparison, Aldous and Pitman \cite{AP94} studied the number of cycles of random mappings, and proved that the number of mappings from $[n]$ to $[n]$ that
have $k$ cyclic points is $k k! n^{n-k-1} \binom{n}{k}$. The two numbers display striking similarity. Indeed, some other statistics are also preserved by the combinatorial connections between mappings and parking functions. 

The {\em terminal closers} of a mapping $f_n$ (resp. of a parking function $\pi_n$) are the elements of the maximal initial segment of distinct terms in $f_n$ (resp. $\pi_n$).
That is, $f_n$ has $k$ terminal closers if $f_n(1), \ldots, f_n(k)$ are all distinct, but $f_n(k+1) = f_n(i)$ for some $i \in [k]$. For mappings, a straightforward computation shows that the number of mappings with exactly $k$ cyclic points coincides with the number of mappings with exactly $k$ terminal closers. By Pollak's circle argument, the same conclusion also holds for parking functions. Thus the number of terminal closers and the number of cyclic points of a uniformly random mapping (resp. random parking function) have the same probability mass function, and are thus equal in distribution. 

Since parking functions describe a dynamic parking process, this equality in distribution between terminal closers and cyclic points is even more intriguing than in the context of mappings. Having $k$ terminal closers in a parking function ensures that the first $k$ cars are lucky (a car is {\em lucky} if it parks at its preferred spot) and that the $(k+1)$st car, having the same preference as one of the first $k$ cars, is unlucky and causes a collision. Thus the expected number of lucky cars in a parking function is lower bounded by the expected number of cyclic points, and moreover, our asymptotic results on cyclic points show that, on average, the first order of $\sqrt{n}$ cars are lucky. 

On the other hand, cyclic points are not necessarily lucky and non-cyclic points are not necessarily unlucky. For example if $\pi_6 = (2, 6, 1, 1, 1, 4)$, then $4$ and $6$ are both cyclic points but neither car $4$ nor car $6$ are lucky, while $3$ is a non-cyclic point but is a lucky car. Intuitively, we would think that disjoint components in the digraph representation of a parking function correspond to non-interacting segments in the parking process, though this does not hold either. For example if $\pi_5=(1, 2, 1, 1, 2)$, then there are two components, $\{1, 3, 4\}$ and $\{2, 5\}$, but clearly the spots where cars $3$, $4$, and $5$ end up parking depend on the preferences of both cars $1$ and $2$. We hope to address these mysteries and more in future work.

We note that there is an alternate interpretation of parking function digraphs, investigated in \cite{KY18} and \cite{LP16}, by considering the vertices as the parking spaces and the directed edges as the one-way streets. Our investigation instead follows the classical interpretation of digraphs as done for mappings and permutations.

\section{Preliminaries} \label{Sec: Prelims}

In this section, we introduce the notations, definitions, and auxiliary results to be used throughout the paper. 

\subsection{Notations and definitions}

Let $a_n, b_n$ be two sequences. If $\lim_{n \to \infty} \frac{a_n}{b_n} = 1$, then we write $a_n\sim b_n$ and say that $a_n$ is {\em asymptotic} to $b_n$. If $\lim_{n \to \infty} \frac{a_n}{b_n} = 0$, then we write $a_n = o(b_n)$ and say that $a_n$ is {\em little-o} of $b_n$. Convergence in distribution and convergence in probability are denoted by $\xrightarrow{D}$ and $\xrightarrow{\PR}$, respectively.
If $X$ is a random variable distributed as $\nu$, then we write $X \sim \nu$. The standard normal random variable is denoted by $\Nc(0,1)$.  

We now introduce some special functions that are of use in our study of the cycle structure. Define the integral $E_1(x) = \int_x^\infty t^{-1} e^{-t} \, dt$, which is related to the {\em exponential integral} by $E_1(x)=-Ei(-x)$. The {\em $r$th generalized Dickman function}, $\rho_r$, is a function which satisfies the delay differential equation
\begin{align} \label{DDE}
    x\rho_r'(x) + \rho_r(x-1) = \rho_{r-1}(x-1), \qquad \text{for $x > 1$ and $r \geq 1$},
\end{align}
with initial conditions $\rho_r(x) = 1$ for $0 \leq x \leq 1$ and $\rho_0 = 0$. The special case $r = 1$ is called the {\em Dickman function} \cite{Dic30}, which appears throughout number theory. 

\subsection{Auxiliary results}

We state some auxiliary results for random permutations that we will use later in the paper. The first is Goncharov's central limit theorem for the number of cycles of a random permutation. 

\begin{proposition}[\cite{Gon44}] \label{GoncharovCLT}
    Let $K_n(\sigma_n)$ be the number of cycles of a uniformly random permutation $\sigma_n \in S_n$. Then the mean and variance of $K_n(\sigma_n)$ are respectively given by
    \begin{align}
    \E(K_n(\sigma_n))=\sum_{k=1}^n \frac{1}{k}, \qquad \VarR(K_n(\sigma_n))=\sum_{k=1}^n \left(\frac{1}{k}-\frac{1}{k^2}\right).   
    \end{align}
    Moreover, asymptotically we have
    \begin{align}
        \frac{K_n(\sigma_n) - \log n}{\sqrt{\log n}} \xrightarrow{D} \Nc(0,1)
    \end{align}
    as $n \to \infty$.
\end{proposition}

We will also need the following result on the asymptotic mean length of the $r$th longest cycle of a random permutation, due to Shepp and Lloyd and recounted by Purdom and Williams. 

\begin{proposition}[\cite{SL66, PW68}] \label{Shepp-LloydLimit}
    Let $L_r(\sigma_n)$ be the length of the $r$th longest cycle of a uniformly random permutation $\sigma_n \in S_n$, defined to be $0$ if the number of cycles $K_n(\sigma_n)$ is less than $r$. Then
    \begin{align}
    E[L_r(\sigma_n)] = (G_{r,1} + \epsilon_{r,1,n})n,
    \end{align}
    where $G_{r,1} = \frac{1}{\Gamma(r)} \int_0^\infty E_1(x)^{r-1} e^{-E_1(x) - x} \, dx$, $\Gamma(r)$ is the Gamma function, $E_1(x)$ is the exponential integral, and $\lim_{n \to \infty} \epsilon_{r,1,n} = 0$. Therefore 
    \begin{align}
        \frac{\E[L_r(\sigma_n)]}{n} \to G_{r,1} 
    \end{align}
    as $n \to \infty$.  
\end{proposition}

The constant $G_{1,1}$ is called the Golomb-Dickman constant \cite{Gol64} and is approximately $0.6243$. Proposition \ref{Shepp-LloydLimit} thus says, for example, that the average length of the longest cycle of a random permutation of size $n$ is asymptotically $0.6243n$. 

Finally we will use the following result due to Knuth and Trabb Pardo on the asymptotic cumulative distribution function of the scaled $r$th longest cycle of a random permutation. Though their paper \cite{KT76} is mainly focused on large prime factors of large numbers, the discussion in Section 10  is about large cycles of random permutations. As Knuth and Trabb Pardo pointed out, the distribution of the number of digits in the prime factors of a random $m$-digit number is approximately the same as the distribution of the cycle lengths in a random permutation of $m$ elements, and the two topics are thus related.

\begin{proposition} [\cite{KT76}] \label{LongCyclesDickman}
Let $L_r(\sigma_n)$ be the length of the $r$th longest cycle of a uniformly random permutation $\sigma_n \in S_n$, defined to be $0$ if the number of cycles $K_n(\sigma_n)$ is less than $r$. Then 
\begin{align}
    \PR\left(\frac{L_r(\sigma_n)}{n} \leq y \right) \to \rho_r(1/y)
\end{align}
as $n \to \infty$, where $\rho_r$ is the $r$th generalized Dickman function.
\end{proposition}

\section{Distribution of cyclic points}
\label{sec:dist}
In this section, we count the number of classical and prime parking functions of length $n$ with exactly $k$ cyclic points. We then use this to obtain the asymptotic distribution of the number of cyclic points in a uniformly random parking function.

\begin{theorem} \label{classical-cyclic}
Let $1\leq k \leq n$. Let $\vert \PF_n^{(k)} \vert$ denote the number of parking functions of length $n$ with $k$ cyclic points. Then
\begin{align}
|\PF_n^{(k)}| = \begin{cases}
n! & k = n, \\
\binom{n+1}{k} k k! (n+1)^{n-k-2} & 1 \leq k < n.
\end{cases}
\end{align}
\end{theorem}

\begin{proof}
We note that a parking function of length $n$ has $k=n$ cyclic points if and only if the parking preferences constitute a permutation of $1, \dots, n$, hence $|\PF_n^{(n)}|=n!$. Now for the general case $1\leq k<n$, we denote by $k_i(\pi_n)$ the number of $i$-cycles in the digraph representation of $\pi_n$. Since there are $k$ cyclic points, we must have $\sum_{i=1}^n k_i(\pi_n) i=k$. We will write $k_i$ instead of $k_i(\pi_n)$ for notational convenience when it is clear from context. Define $M(m, k_1, \dots, k_n)$ to be the multinomial coefficient
$$M(m, k_1, \dots, k_n)=\binom{m}{\underbracket[0.5pt]{1, \dots, 1}_{k_1 \hspace{.1cm} \text{$1$'s}}, \dots, \underbracket[0.5pt]{n, \dots, n}_{k_n \hspace{.1cm} \text{$n$'s}}, m-\sum_{i=1}^n k_i i}.$$

We use an extension of the ideas in Rubey-Yin \cite{RY25}[Theorem 2.1, Theorem 2.3] and highlight key steps in the derivation. Let $1\leq m \leq n$. For distinct $1\leq i_1, \dots, i_m\leq n$, let $$A_{(i_1, \dots, i_m)}=\{\pi_n \in \PF_n: \pi_n(i_1)=i_2, \pi_n(i_2)=i_3, \dots, \pi_n(i_{m-1})=i_m, \pi_n(i_m)=i_1\}.$$ The number of parking functions $\pi_n$ of length $n$ where $G_{\pi_n}$ has $k_i$ cycles of length $i$ for every $i \in [n]$ may be represented by the following formula:
\begin{multline}\label{counts-0}
\left(\prod_{m=1}^n \left((m-1)!\right)^{k_m}\right) \sum_{\substack{(i^j_1, \dots, i^j_m) \\ 1\leq j\leq k_m; \, 1\leq m\leq n}} \Bigg \vert \left(\bigcap_{1\leq j\leq k_m; \, 1\leq m\leq n} A_{(i^j_1, \dots, i^j_m)} \right)  \\ \bigcap \left(\bigcap_{\substack{(i^0_1, \dots, i^0_{r}) \text{ and } (i^j_1, \dots, i^j_m) \text{ are non-overlapping} \\ 1\leq j\leq k_m; \, 1\leq m\leq n; \, 1\leq r\leq n}} A^c_{(i^0_1, \dots, i^0_{r})} \right) \Bigg \vert,
\end{multline}
where the sum is over all possible $k_m$ $m$-cycles for every $m \in [n]$, and all the cycles are pairwise non-overlapping, i.e., we are choosing altogether $\sum_{m=1}^n k_m m=k$ distinct points from $[n]$ and assign them to different cycles. To avoid overcounting, for every $m \in [n]$, we arrange the $m$-cycles $(i^j_1, \dots, i^j_m)$ so that for each $1\leq j\leq k_m$, $i^j_1<\cdots<i^j_m$ (the entries of the cycles are in increasing order, which is allowed by symmetry of parking coordinates) and further $i^1_1<\cdots<i^{k_m}_1$ (the $m$-cycles are listed by increasing order of their first entries). This convention introduces an additional scalar factor $(m-1)!$ for each such $m$-cycle, which accounts for the number of ways of arranging an $m$-cycle with distinct entries. We use $A^c$ to denote the complement of the set $A$. We restrict ourselves to sets $A^c_{(i^0_1, \dots, i^0_{r})}$ where $1\leq r\leq n$ and $(i_1^0, \dots, i^0_{r})$ is non-overlapping with any $(i^j_1, \dots, i^j_m)$ where $1\leq j\leq k_m$ and $1\leq m\leq n$, because $A_{(i^j_1, \dots, i^j_m)} \subseteq A^c_{(i^0_1, \dots, i^0_{r})}$ if some entries of $(i^0_1, \dots, i^0_{r})$ and $(i^j_1, \dots, i^j_m)$ coincide (unless $(i^0_1, \dots, i^0_{r})$ constitutes a circular rotation of $(i^j_1, \dots, i^j_m)$ in which case $A_{(i^j_1, \dots, i^j_m)}=A_{(i^0_1, \dots, i^0_{r})}$). Thus (\ref{counts-0}) counts parking functions of length $n$ with exactly $k_m$ cycles of length $m$ for every $m \in [n]$.

By De Morgan's law and the inclusion-exclusion principle, (\ref{counts-0}) equals
\begin{multline}\label{counts-1}
\sum_{(\ell_1, \dots, \ell_n): \sum_{m=1}^n (k_m+\ell_m)m \leq n} (-1)^{\sum_{m=1}^n \ell_m} \ \prod_{m=1}^n \left((m-1)!\right)^{k_m+\ell_m} \\
\times \sum_{\substack{(i^j_1, \dots, i^j_m) \\ 1\leq j\leq k_m; \, 1\leq m\leq n}} \sum_{\substack{(s^t_1, \dots, s^t_m) \\ 1\leq t\leq \ell_m; \, 1\leq m\leq n}}
\Bigg \vert \left(\bigcap_{1\leq j\leq k_m; \, 1\leq m\leq n} A_{(i^j_1, \dots, i^j_m)} \right) \bigcap \left(\bigcap_{1\leq t\leq \ell_m; \, 1\leq m\leq n} A_{(s^t_1, \dots, s^t_m)} \right) \Bigg \vert,
\end{multline}
where the middle sum is over all possible $k_m$ $m$-cycles for every $m \in [n]$, and the last sum is over all possible $\ell_m$ $m$-cycles for every $m \in [n]$, and all the cycles are pairwise non-overlapping, i.e., we are choosing altogether $\sum_{m=1}^n (k_m+\ell_m)m \leq n$ distinct points from $[n]$ and assign them to different cycles. As previously, to avoid overcounting, for every $m \in [n]$, we arrange the $m$-cycles $(i^j_1, \dots, i^j_m)$ so that for each $1\leq j\leq k_m$, $i^j_1<\cdots<i^j_m$ and further $i^1_1<\cdots<i^{k_m}_1$, and similarly for each $1\leq t\leq \ell_m$, $s^t_1<\cdots<s^t_m$ and further $s^1_1<\cdots<s^{\ell_m}_1$. This convention introduces an additional scalar factor $(m-1)!$ for each such $m$-cycle.

We use symmetry of parking coordinates to rewrite the complicated sums in (\ref{counts-1}):
\begin{align}\label{counts-2}
&\sum_{\substack{(i^j_1, \dots, i^j_m) \\ 1\leq j\leq k_m; \, 1\leq m\leq n}} \sum_{\substack{(s^t_1, \dots, s^t_m) \\ 1\leq t\leq \ell_m; \, 1\leq m\leq n}}
\Bigg \vert \left(\bigcap_{1\leq j\leq k_m; \, 1\leq m\leq n} A_{(i^j_1, \dots, i^j_m)} \right) \bigcap \left(\bigcap_{1\leq t\leq \ell_m; \, 1\leq m\leq n} A_{(s^t_1, \dots, s^t_m)} \right) \Bigg \vert \notag \\
=& \frac{1}{\prod_{m=1}^n k_m!\ell_m!} \Bigg\vert\Bigg\{\pi_n \in \PF_n: \notag \\ & \hspace{-.2cm} \quad
\substack{\pi_n(1), \dots, \pi_n(k_1+\ell_1), \dots, \pi_n\left(\sum_{m=1}^2 (k_m+\ell_m)m\right), \dots, \pi_n\left(\sum_{m=1}^n (k_m+\ell_m)m\right) \text{ are pairwise distinct} \\ \pi_n\left(\sum_{m=1}^{r-1} (k_m+\ell_m)m+j\right)<\pi_n\left(\sum_{m=1}^{r-1} (k_m+\ell_m)m+j+1\right)<\cdots<\pi_n\left(\sum_{m=1}^{r-1} (k_m+\ell_m)m+j+r-1\right) \\ j=1, r+1, 2r+1, \dots, (k_r+\ell_r-1)r+1; \, 2\leq r\leq n} \Bigg\} \Bigg\vert.
\end{align}

Next we apply an extension of Pollak's circle argument. Add an additional parking spot $n+1$, and arrange the spots in a circle. Allow $n+1$ also as a preferred spot. We first select $k_m+\ell_m$ $m$-cycles on the circle for every $m \in [n]$ and designate these spots to the first $\sum_{m=1}^n (k_m+\ell_m)m$ cars. Then for the remaining $n-\sum_{m=1}^n (k_m+\ell_m)m$ cars, there are $(n+1)^{n-\sum_{m=1}^n (k_m+\ell_m)m}$ possible preference sequences. Out of the $n+1$ rotations for any preference sequence, only one rotation becomes a valid parking function. Standard circular symmetry argument yields (\ref{counts-2}) equals
\begin{align*}
\frac{1}{\prod_{m=1}^n k_m!\ell_m!} M(n+1, k_1+\ell_1, \dots, k_n+\ell_n) (n+1)^{n-1-\sum_{m=1}^n (k_m+\ell_m)m}.
\end{align*}

Putting everything together,
\begin{align}\label{long-cyclic}
&|\PF_n^{(k)}| = \sum_{(k_1, \dots, k_n): \sum_{i=1}^n k_i i=k} \
\sum_{(\ell_1, \dots, \ell_n): \sum_{i=1}^n (k_i+\ell_i)i \leq n} (-1)^{\sum_{i=1}^n \ell_i} \ \frac{\prod_{i=1}^n \left((i-1)!\right)^{k_i+\ell_i}}{\prod_{i=1}^n k_i!\ell_i!} \notag \\
& \hspace{3cm} \times  M(n+1, k_1+\ell_1, \dots, k_n+\ell_n) (n+1)^{n-1-\sum_{i=1}^n (k_i+\ell_i)i} \notag \\
&=\sum_{(k_1, \dots, k_n): \sum_{i=1}^n k_i i=k} \frac{\prod_{i=1}^n \left((i-1)!\right)^{k_i}}{\prod_{i=1}^n k_i!} M(n+1, k_1, \dots, k_n) \notag \\ \
& \hspace{.5cm} \times \sum_{(\ell_1, \dots, \ell_n): \sum_{i=1}^n \ell_i i \leq n-k} (-1)^{\sum_{i=1}^n \ell_i} \frac{\prod_{i=1}^n \left((i-1)!\right)^{\ell_i}}{\prod_{i=1}^n \ell_i!} M(n+1-k, \ell_1, \dots, \ell_n) (n+1)^{n-1-k-\sum_{i=1}^n \ell_i i} \notag \\
&=k! \binom{n+1}{k} (n+1)^{n-k-2} \sum_{(k_1, \dots, k_n): \sum_{i=1}^n k_i i=k} \frac{1}{\prod_{i=1}^n i^{k_i} k_i!} \notag \\
&\hspace{1.5cm} \times \sum_{(\ell_1, \dots, \ell_n): \sum_{i=1}^n \ell_i i \leq n-k} (-1)^{\sum_{i=1}^n \ell_i} \frac{(n+1-k)!}{\prod_{i=1}^n i^{\ell_i} \ell_i! (n+1-k-\sum_{i=1}^n \ell_i i)!} (n+1)^{1-\sum_{i=1}^n \ell_i i}.
\end{align}

For parking functions $\pi_n \in \PF_n$, having $k$ cyclic points forces $k_i=0$ for all $k<i\leq n$ and so $\prod_{i=1}^k i^{k_i} k_i!=\prod_{i=1}^n i^{k_i} k_i!$. From Riordan \cite[page 67]{Riordan}, $k! /\prod_{i=1}^k i^{k_i} k_i!$ records the number of permutations in $S_k$ of cycle type $(1^{k_1}, 2^{k_2}, \dots)$\footnote{One way to see this is to insert pairs of parentheses from left to
right into the one-line notation for the permutation according to its cycle type, and then note that for each $i$ we can permute the $k_i$ distinct cycles of length-$i$ and a length-$i$ cycle can have $i$ possible starting points.}. This implies that
\begin{align}\label{eq:cycle_type}
\sum_{(k_1, \dots, k_n): \sum_{i=1}^n k_i i=k} \frac{1}{\prod_{i=1}^n i^{k_i} k_i!}=1.
\end{align}
The first sum in (\ref{long-cyclic}) thus disappears. 

Let $j=\sum_{i=1}^n \ell_i i$. Then the second sum in (\ref{long-cyclic}) is over all partitions $p$ of $j$ for $0\leq j\leq n-k$. For notational convenience, denote by $\ell(p)=\sum_{i=1}^n \ell_i$ and $z_p=\prod_{i=1}^n i^{\ell_i} \ell_i!$ as in Macdonald \cite[page 1, page 24]{Mac95}, Equation (\ref{long-cyclic}) becomes
\begin{align}\label{last-sum}
&k! \binom{n+1}{k} (n+1)^{n-k-2} \sum_{j=0}^{n-k} \sum_{p\vdash j} (-1)^{\ell(p)} \frac{(n+1-k)!}{z_p (n+1-k-j)!} (n+1)^{1-j} \notag \\
&=k! \binom{n+1}{k} (n+1)^{n-k-2} \sum_{j=0}^{n-k} \frac{(n+1-k)!}{(n+1-k-j)!} (n+1)^{1-j} \sum_{p\vdash j} (-1)^{\ell(p)} \frac{1}{z_p}.
\end{align}

We note that the last sum in (\ref{last-sum}) is equal to $1$ if $j=0$, $-1$ if $j=1$, and $0$ if $j\geq 2$ (because there are as many even permutations in $S_j$ as odd permutations). Since $k<n$, $j=0$ and $1$ are both valid, (\ref{last-sum}) is reduced to
\begin{align}
&k! \binom{n+1}{k} (n+1)^{n-k-2} \sum_{j=0}^1 \frac{(n+1-k)!}{(n+1-k-j)!}(n+1)^{1-j}(-1)^j  \\
&=\binom{n+1}{k} k k! (n+1)^{n-k-2}. \qedhere
\end{align}
\end{proof}

\begin{remark}\label{rk:exp_cyclic}
Implicitly, the proof of Theorem \ref{classical-cyclic} establishes a more refined result: The number of parking functions $\pi_n$ of length $n$ where $G_{\pi_n}$ has $k_i$ cycles of length $i$ for every $i \in [n]$ is given by
\begin{equation*}
\begin{cases}
\frac{1}{\prod_{i=1}^n i^{k_i} k_i!} n! & \text{ if } k = n, \\
\frac{1}{\prod_{i=1}^n i^{k_i} k_i!}\binom{n+1}{k} k k! (n+1)^{n-k-2} & \text{ if } 1 \leq k < n,
\end{cases}
\end{equation*}
where $k:=\lambda_n(\pi_n)=\sum_{i=1}^n k_i i$ is the number of cyclic points of $\pi_n$. Here we are not summing over different cycle types $(1^{k_1}, 2^{k_2}, \dots)$ of $S_k$ as in (\ref{eq:cycle_type}) and so the scalar factor $\frac{1}{\prod_{i=1}^n i^{k_i} k_i!}$ stays. This further implies that the cyclic points of $\pi_n$ form a uniform permutation of size $\lambda_n(\pi_n)$.
\end{remark}

\begin{theorem}\label{prime-cyclic}
Let $1\leq k \leq n-1$. Let $\vert \PPF_n^{(k)} \vert$ denote the number of prime parking functions of length $n$ with $k$ cyclic points. Then
\begin{align}
|\PPF_n^{(k)}| = kk!\binom{n-1}{k}(n-1)^{n-k-2}.
\end{align}
\end{theorem}

\begin{proof}
The proof proceeds similarly as in the proof of Theorem \ref{classical-cyclic} for classical parking functions, except that our circular symmetry argument is now applied to a circle with $n-1$ spots instead of $n+1$ spots as in the classical case. Note that for prime parking functions it is impossible to have $n$ cyclic points, since by definition it must have at least two coordinates equalling $1$, and so we must have $1 \leq k\leq n-1$. Using the same reasoning as in the proof of Theorem \ref{classical-cyclic} (with minor adaptation), we will in the end arrive at the following equation:
\begin{align}
&|\PPF_n^{(k)}|=k! \binom{n-1}{k} (n-1)^{n-k-2} \sum_{j=0}^1 \frac{(n-1-k)!}{(n-1-k-j)!}(n-1)^{1-j}(-1)^j \\
&=k k! \binom{n-1}{k} (n-1)^{n-k-2}. \qedhere
\end{align}
\end{proof}

\begin{remark}
While we present the enumeration result for cyclic points in prime parking functions above, the asymptotic investigation for prime parking functions will be left later for Section \ref{sec:prime} as this investigation for prime parking functions is largely in parallel with that of classical parking functions. The rest of Section \ref{sec:dist} as well as Sections \ref{sec:trio} and \ref{sec:len} will focus exclusively on classical parking functions.
\end{remark}

\vskip.1truein

Recall from the introduction that $\lambda_n = \lambda_n(\pi_n)$ denotes the number of cyclic points in a uniformly random parking function $\pi_n \in \PF_n$. We now turn to the asymptotic behavior of $\lambda_n$, starting with the probability mass function. 

\begin{lemma} \label{PFcyclicPMF}
Let $\lambda_n(\pi_n)$ be the number of cyclic points of a uniformly random parking function $\pi_n \in \PF_n$. Then
\begin{align}
\PR(\lambda_n(\pi_n) = k) = \begin{cases}
\frac{n!}{(n+1)^{n-1}} & k = n, \\
\frac{k n!}{(n+1)^k (n+1-k)!} & 1 \leq k < n.
\end{cases}
\end{align}
\end{lemma}

\begin{proof}
Let $|\PF_n^{(k)}|$ denote the number of parking functions of length $n$ with exactly $k$ cyclic points. Then $\PR(\lambda_n(\pi_n) = k) = \frac{|\PF_n^{(k)}|}{(n+1)^{n-1}}$. Our claim readily follows from Theorem \ref{classical-cyclic}.
\end{proof}

Our next result establishes a local limit theorem for $\lambda_n(\pi_n)$. In particular, it shows that the number of cyclic points, scaled by $\sqrt{n}$, converges in distribution to a Rayleigh random variable. 

\begin{proposition} \label{PFcyclicPDF}
Let $\lambda_n(\pi_n)$ be the number of cyclic points of a uniformly random parking function $\pi_n \in \PF_n$. Let $\Lambda_n = \frac{\lambda_n(\pi_n)}{\sqrt{n}}$ be the scaled number of cyclic points. If $k/\sqrt{n} \to x > 0$ as $n \to \infty$, then 
\begin{align}
    \PR(\lambda_n(\pi_n) = k) = \frac{1}{\sqrt{n}} x e^{-x^2/2}\left(1 + \frac{x(3-x^2)}{2\sqrt{n}} + O(n^{-1}) \right),
\end{align}
so that $\Lambda_n \xrightarrow{D} \Lambda$ as $n \to \infty$, where $\Lambda$ is distributed as a $\Rayleigh(1)$ random variable.
\end{proposition}

\begin{proof}
Let $k = x\sqrt{n}$ with $x > 0$, and note that asymptotically, $k < n$. By Lemma \ref{PFcyclicPMF} and Stirling approximation, 
\begin{align}
\begin{split}
&\PR(\lambda_n(\pi_n) = k) = \frac{x\sqrt{n} n!}{(n+1)^{x\sqrt{n}} (n+1-x\sqrt{n})!} \\
&= \frac{x\sqrt{n} \sqrt{2\pi} n^{n+\frac{1}{2}} e^{-n}}{(n+1)^{x\sqrt{n}} \sqrt{2\pi} (n+1-x\sqrt{n})^{n+\frac{3}{2}-x\sqrt{n}} e^{-(n+1-x\sqrt{n})}} (1 + O(n^{-1})) \\
&= x \exp\left(1-x\sqrt{n}+(n+1)\log n-(n+\frac{3}{2}-x\sqrt{n})\log(n+1-x\sqrt{n})-x\sqrt{n}\log(n+1)\right) (1 + O(n^{-1})) \\
&= x \exp\left(1-x\sqrt{n}+(n+1)\log n-(n+\frac{3}{2}-x\sqrt{n})\left[ \log n + \frac{1}{n} - \frac{x}{\sqrt{n}} - \frac{x^2}{2n} + O(n^{-3/2}) \right] \right. \\
& \qquad \qquad \left. -x\sqrt{n}\left[ \log n + \frac{1}{n} + O(n^{-2}) \right]\right) (1 + O(n^{-1})) \\
&= \frac{1}{\sqrt{n}} x e^{-x^2/2}\left(1 + \frac{x(3-x^2)}{2\sqrt{n}} + O(n^{-1}) \right).
\end{split}
\end{align}
We see that $\sqrt{n} P(\lambda_n(\pi_n) = k) \to x e^{-x^2/2}$ pointwise as $n \to \infty$. Indeed, this pointwise convergence as $n \to \infty$ is uniform with respect to $k$ in any interval of the form $0<a\leq k/\sqrt{n}\leq b<\infty$. Since $x e^{-x^2/2}$ is the density function of a $\Rayleigh(1)$ random variable, the claimed convergence in distribution follows from Scheff\'{e}'s lemma.
\end{proof}

Using this, we can compute the asymptotic mean number of cyclic points. 

\begin{lemma} \label{PFcyclicmean}
Let $\lambda_n(\pi_n)$ be the number of cyclic points of a uniformly random parking function $\pi_n \in \PF_n$. Then 
\begin{align}
\E(\lambda_n(\pi_n)) \sim \sqrt{\frac{\pi n}{2}}. 
\end{align}
\end{lemma}

\begin{proof}
From the explicit error bound in Proposition \ref{PFcyclicPDF}, if $n \to \infty$, then uniformly with respect to integers $k$ such that $x = \frac{k}{\sqrt{n}}$ lies in any interval of the form $0 < x \leq C < \infty$, where $C$ is some constant, we have that $\sqrt{n}P\left( \frac{\lambda_n(\pi_n)}{\sqrt{n}} = x\right) = xe^{-x^2/2}(1 + o(1))$, where $o(1) \to 0$ as $n \to \infty$. Fix $0 < M < \infty$. By truncation, we can write the limit of the expectation as 
\begin{align}
\begin{split} \label{cyclicmeantrunc}
   \lim_{n \to \infty} \E\left( \frac{\lambda_n(\pi_n)}{\sqrt{n}} \right) &= \lim_{n \to \infty} \E\left( \frac{\lambda_n(\pi_n)}{\sqrt{n}} \I_{\left\{ \frac{\lambda_n(\pi_n)}{\sqrt{n}} \leq M \right\}} \right) + \lim_{n \to \infty} \E\left( \frac{\lambda_n(\pi_n)}{\sqrt{n}} \I_{\left\{ \frac{\lambda_n(\pi_n)}{\sqrt{n}} > M \right\}} \right).
\end{split}
\end{align}
The second term of (\ref{cyclicmeantrunc}) is uniformly bounded in $n$ for $k$ in the interval $(M\sqrt{n}, n]$ and converges to $0$ as $M \to \infty$. 

Now let $x_k := \frac{k}{\sqrt{n}}$. By Riemann sum approximation and applying the asymptotic form of Proposition \ref{PFcyclicPDF}, the first term of (\ref{cyclicmeantrunc}) converges as
\begin{align}
\begin{split}
\lim_{n \to \infty} \E\left( \frac{\lambda_n(\pi_n)}{\sqrt{n}} \I_{\left\{ \frac{\lambda_n(\pi_n)}{\sqrt{n}} \leq M \right\}} \right) &= \lim_{n \to \infty} \sum_{1 \leq k \leq M\sqrt{n}} \frac{k}{\sqrt{n}} P(\lambda_n(\pi_n) = k) \\
&= \lim_{n \to \infty} \sum_{1 \leq k \leq M\sqrt{n}} x_k \cdot x_k e^{-x_k^2/2}(1 + o(1)) \frac{1}{\sqrt{n}} \\
&= \int_0^M x \cdot xe^{-x^2/2} \, dx.
\end{split}
\end{align}
Therefore we have that 
\begin{align}
    \lim_{n \to \infty} \E\left( \frac{\lambda_n(\pi_n)}{\sqrt{n}} \right) = \int_0^M x \cdot xe^{-x^2/2} \, dx. 
\end{align}
Since $M > 0$ was arbitrary, sending $M \to \infty$ and noting that $\int_0^\infty x^2 e^{-x^2/2} \, dx = \sqrt{\frac{\pi}{2}}$ completes the proof. 
\end{proof}

Since we will need it later, we also record the following result. The proof follows by a similar argument as that of Lemma \ref{PFcyclicmean}. 

\begin{lemma} \label{PFcyclicmomentconvergence}
Let $\Lambda_n = \frac{\lambda_n(\pi_n)}{\sqrt{n}}$ be the scaled number of cyclic points of a uniformly random parking function $\pi_n \in \PF_n$. Then for any $r \in (-2, \infty)$,
\begin{align}
\lim_{n \to \infty} \E(\Lambda_n^r) = \E(\Lambda^r) = 2^{r/2} \Gamma\left(1 + \frac{r}{2}\right). 
\end{align}
\end{lemma}

\begin{proof}
This follows from Proposition \ref{PFcyclicPDF}, truncation, and standard facts about moments of the Rayleigh distribution.
\end{proof}

\section{Trio of limit theorems for the number of cycles}
\label{sec:trio}

In this section, we establish the classical trio of limit theorems (law of large numbers, central limit theorem, large deviation principle) for the number of cycles in a random parking function. We present our main results first.

\begin{theorem}[Law of large numbers] \label{PFcyclesLLN}
Let $K_n(\pi_n)$ be the number of cycles of a uniformly random parking function $\pi_n \in \PF_n$. Then
\begin{align}
    \frac{K_n(\pi_n)}{\frac{1}{2}\log n} \xrightarrow{\PR} 1
\end{align}
as $n \to \infty$.
\end{theorem}

\begin{theorem}[Central limit theorem] \label{PFcyclesCLT}
Let $K_n(\pi_n)$ be the number of cycles of a uniformly random parking function $\pi_n \in \PF_n$. Then
\begin{align}
\frac{K_n(\pi_n) - \frac{1}{2}\log n}{ \sqrt{\frac{1}{2} \log n} } \xrightarrow{D} \mathcal{N}(0,1),
\end{align}
as $n \to \infty$, where $\mathcal{N}(0,1)$ is a standard normal random variable. 
\end{theorem}

\begin{theorem}[Large deviation principle] \label{PFcyclesLDP}
Let $K_n(\pi_n)$ be the number of cycles of a uniformly random parking function $\pi_n \in \PF_n$. Then the family $\left\{ \frac{K_n(\pi_n)}{\log n}\right\}_{n \geq 1}$ satisfies a large deviation principle with speed $\log n$ and rate function
\begin{align}
I(x) = \begin{cases}
x\log\left( 2x \right) - x + \frac{1}{2} & x \geq 0, \\
\infty & x < 0. 
\end{cases}
\end{align}
\end{theorem}

To prove these results, we start by computing the asymptotic mean and variance of the number of cycles in a uniformly random parking function. The key idea is to study $(K_n(\pi_n) \mid \lambda_n)$, the number of cycles in a uniformly random parking function conditioned to have $\lambda_n(\pi_n)$ cyclic points. 

\begin{lemma} \label{PFcyclesmean}
The mean number of cycles is
\begin{align}
\E(K_n(\pi_n)) = \frac{1}{2}(\log(2n) + \gamma) + o(1),
\end{align}
where $\gamma$ is the Euler-Mascheroni constant. 
\end{lemma}

\begin{remark}\label{rK: gamma}
The Euler-Mascheroni constant $\gamma$ may be defined in various ways. The proof of Lemma \ref{PFcyclesmean} will utilize two of these equivalent definitions, which we present below.

\noindent (1) $\gamma$ is the limiting difference between the harmonic series and the natural logarithm:
\begin{align}
\gamma=\lim_{n \rightarrow \infty} \left(-\log n + \sum_{k=1}^n \frac{1}{k}\right).
\end{align}

\noindent (2) Recall that the Gamma function is defined as
\begin{align}
\Gamma(z)=\int_0^\infty x^{z-1} e^{-x} dx,
\end{align}
with derivatives given by
\begin{align}
\Gamma'(z)=\int_0^\infty x^{z-1} e^{-x} \log x dx \text{ and } \Gamma''(z)=\int_0^\infty x^{z-1} e^{-x} \log^2 x dx.
\end{align}
We have $\Gamma'(1)=-\gamma$ and $\Gamma''(1)=\gamma^2+\pi^2/6$.
\end{remark}


\begin{proof}[Proof of Lemma \ref{PFcyclesmean}]
Observe that the $\lambda_n(\pi_n)$ cyclic points of $\pi_n$ form a random permutation of length $\lambda_n(\pi_n)$, cf. Remark \ref{rk:exp_cyclic}. Using Proposition \ref{GoncharovCLT} for the number of cycles in a random permutation of size $n$, we have that the mean number of cycles is $\sum_{k=1}^n 1/k$, and thus $\log n + \gamma + o(1)$ by Remark \ref{rK: gamma}(1). As in the proof of Lemma \ref{PFcyclicmean}, we proceed by direct computation (which is again based on Proposition \ref{PFcyclicPDF}).
\begin{align}
\begin{split}
& \lim_{n \rightarrow \infty} \E\left(\log \left(\frac{\lambda_n(\pi_n)}{\sqrt{n}}\right)\right)=\E(\log \Lambda) \\
= & \int_0^\infty \log x \cdot x e^{-\frac{x^2}{2}} dx = \frac{1}{2} \int_0^\infty \log(2t) e^{-t} dt \\
= &\frac{1}{2} \left(\log 2 + \int_0^\infty \log t e^{-t} dt\right) =\frac{1}{2} \left(\log 2-\gamma \right),
\end{split}
\end{align}
where in the middle equality we applied a change of variables $x=\sqrt{2t}$ and in the last equality we applied Remark \ref{rK: gamma}(2). By the law of total expectation, conditioning on $\lambda_n(\pi_n)$, this further implies that
\begin{align}
\E(K_n(\pi_n)) &= \E[\E(K_n(\pi_n) \mid \lambda_n)] = \E\left[ \log(\lambda_n(\pi_n)) + \gamma + o(1) \right] \nonumber \\
= & \log(\sqrt{n}) + \frac{1}{2}(\log 2 - \gamma) + \gamma + o(1) \\
= & \frac{1}{2}(\log(2n) + \gamma) + o(1). \nonumber \qedhere
\end{align}
\end{proof}

The asymptotic variance of $K_n(\pi_n)$ may be computed by the same conditioning argument as in the proof of Lemma \ref{PFcyclesmean}, except that the computation is more involved.

\begin{lemma} \label{PFcyclesvar}
The variance of the number of cycles is
\begin{align}
\Var(K_n(\pi_n)) = \frac{1}{2}(\log(2n) + \gamma)-\frac{\pi^2}{8} + o(1),
\end{align}
where $\gamma$ is the Euler-Mascheroni constant. 
\end{lemma}

\begin{proof}
Using Proposition \ref{GoncharovCLT}, we have that the variance of the number of cycles in a random permutation of size $n$ is
\begin{align}
\sum_{k=1}^n \left(\frac{1}{k}-\frac{1}{k^2}\right)=\log n + \gamma - \zeta(2) + o(1)=\log n + \gamma - \frac{\pi^2}{6} +o(1),
\end{align}
where we used Remark \ref{rK: gamma}(1) and properties of the Riemann zeta function $\zeta(\cdot)$.
By Proposition \ref{PFcyclicPDF},
\begin{align}
\begin{split}
& \lim_{n \rightarrow \infty} \E\left(\log^2 \left(\frac{\lambda_n(\pi_n)}{\sqrt{n}}\right)\right)=\E(\log^2 \Lambda) \\
= & \int_0^\infty \log^2 x \cdot x e^{-\frac{x^2}{2}} dx = \frac{1}{4} \int_0^\infty \log^2(2t) e^{-t} dt \\
= & \frac{1}{4} \left(\log^2 2 +2\log 2 \int_0^\infty \log t e^{-t} dt + \int_0^\infty \log^2 t e^{-t} dt\right) \\
= & \frac{1}{4} \left(\log^2 2 -2\gamma\log 2 +\gamma^2+\frac{\pi^2}{6} \right),
\end{split}
\end{align}
where in the middle equality we applied a change of variables $x=\sqrt{2t}$ and in the last equality we applied Remark \ref{rK: gamma}(2). Recall from the proof of Lemma \ref{PFcyclesmean} that
\begin{align}
\E(\log \lambda_n(\pi_n))=\log(\sqrt{n})+\frac{1}{2}(\log 2-\gamma).
\end{align}
This says that
\begin{align}
\begin{split}
&\VarR[\log \lambda_n(\pi_n)]=\E [\log^2 \lambda_n(\pi_n)] -\left(\E[\log \lambda_n(\pi_n)]\right)^2 \\
=& 2\log(\sqrt{n}) \left(\log(\sqrt{n})+\frac{1}{2}(\log 2-\gamma)\right) -(\log(\sqrt{n}))^2 \\
& \hspace{1cm}+\frac{1}{4} \left(\log^2 2 -2\gamma\log 2 +\gamma^2+\frac{\pi^2}{6} \right)-\left(\log(\sqrt{n})+\frac{1}{2}(\log 2-\gamma)\right)^2=\frac{\pi^2}{24}.
\end{split}
\end{align}
By the law of total variance, conditioning on $\lambda_n(\pi_n)$, we further have
\begin{align}
\VarR(K_n(\pi_n)) &= \E[\VarR(K_n(\pi_n) \mid \lambda_n(\pi_n))]+\VarR[\E(K_n(\pi_n) \mid \lambda_n(\pi_n))] \nonumber \\
=& \E\left[\log (\lambda_n(\pi_n)) + \gamma - \frac{\pi^2}{6} +o(1)\right]+\VarR \bigg[\log(\lambda_n(\pi_n)) + \gamma + o(1)\bigg] \\
=& \log(\sqrt{n}) + \frac{1}{2}(\log 2 - \gamma)+ \gamma - \frac{\pi^2}{6} + \frac{\pi^2}{24}+o(1) \nonumber \\
= & \frac{1}{2}(\log(2n) + \gamma) - \frac{\pi^2}{8}+o(1), \nonumber
\end{align}
where $\E(K_n(\pi_n) \mid \lambda_n(\pi_n))$ was previously computed in the proof of Lemma \ref{PFcyclesmean}. 
\end{proof}

Using the asymptotic mean and variance established in Lemmas \ref{PFcyclesmean} and \ref{PFcyclesvar}, we are ready to demonstrate the law of large numbers for the number of cycles. 

\begin{proof}[Proof of Theorem \ref{PFcyclesLLN}.]
Consider the decomposition
\begin{align}
    \frac{K_n(\pi_n)}{\frac{1}{2}\log n} - 1 = \frac{K_n(\pi_n) - \frac{1}{2}\log n}{\frac{1}{2}\log n} = \frac{K_n(\pi_n) - \E(K_n(\pi_n))}{\frac{1}{2}\log n} + \frac{\E(K_n(\pi_n)) - \frac{1}{2}\log n}{\frac{1}{2}\log n}. 
\end{align}
The second summand is deterministic and converges to $0$ as $n \to \infty$ by Lemma \ref{PFcyclesmean}. Fix $\epsilon > 0$. By Chebyshev's inequality and Lemma \ref{PFcyclesvar}, 
\begin{align}
\begin{split}
    \PR\left( \left| \frac{K_n(\pi_n) - \E(K_n(\pi_n))}{\frac{1}{2} \log n} \right| > \epsilon \right) &= \PR\left( \left| K_n(\pi_n) - \E(K_n(\pi_n)) \right| > \frac{\epsilon}{2} \log n \right) \\
    &\leq \frac{4 \Var(K_n(\pi_n))}{\epsilon^2 (\log n)^2} \\
    &= \frac{4 \left( \frac{1}{2} (\log(2n) + \gamma) - \frac{\pi^2}{8} + o(1) \right)}{\epsilon^2 (\log n)^2} \\
    &\to 0
    \end{split}
\end{align}
as $n \to \infty$, which implies that the first summand converges to $0$ in probability as $n \to \infty$. The proof follows.
\end{proof}

Next we move on to verify the central limit theorem for the number of cycles. 


\begin{proof}[Proof of Theorem \ref{PFcyclesCLT}]
Let $K_{0,n}$ be the number of cycles in a uniformly random permutation of size $n$. By conditioning on the number of cyclic points $\lambda_n$, we have that
\begin{align} \label{distrep}
    K_n(\pi_n) \stackrel{D}{=} K_{0,\lambda_n(\pi_n)},
\end{align}
where $\stackrel{D}{=}$ denotes equality in distribution. For notational simplicity, let $K_n := K_n(\pi_n)$ and $\lambda_n := \lambda_n(\pi_n)$. 
By Equation (\ref{distrep}), we can write 
\begin{align}
\begin{split}
    \frac{ K_n - \frac{1}{2}\log n }{ \sqrt{\frac{1}{2}\log n } } &\stackrel{D}{=} \frac{ K_{0,\lambda_n} - \frac{1}{2}\log n }{ \sqrt{\frac{1}{2}\log n } } \\
    &= \sqrt{\frac{\log \lambda_n}{ \frac{1}{2}\log n } } \left( \frac{ K_{0,\lambda_n} - \log \lambda_n }{ \sqrt{\log \lambda_n } } \right) + \frac{ \log \lambda_n - \frac{1}{2}\log n }{ \sqrt{\frac{1}{2}\log n } }.
    \end{split}
\end{align}
By Proposition \ref{GoncharovCLT} and the fact that $\lambda_n \to \infty$ almost surely as $n \to \infty$, we have that
\begin{align}
\frac{ K_{0,\lambda_n} - \log \lambda_n }{ \sqrt{\log \lambda_n } } \xrightarrow{D} \mathcal{N}(0,1),
\end{align}
as $n \to \infty$. Moreover, 
\begin{align}
\frac{\log \lambda_n}{ \frac{1}{2}\log n } = \frac{\log(\sqrt{n}) + \log\left( \frac{\lambda_n}{\sqrt{n}} \right)}{\log(\sqrt{n})} = 1 + \frac{\log\left( \frac{\lambda_n}{\sqrt{n}} \right)}{\log(\sqrt{n})} \xrightarrow{D} 1,
\end{align}
where $\frac{\log\left( \frac{\lambda_n}{\sqrt{n}} \right)}{\log(\sqrt{n})} \xrightarrow{D} 0$ as $n \to \infty$ follows from the fact that $\frac{\lambda_n}{\sqrt{n}} \xrightarrow{D} \Lambda$ by Proposition \ref{PFcyclicPDF}. Since convergence in distribution to a constant implies convergence in probability to that same constant, the convergence to $1$ in the above also holds in probability. Thus by Slutsky's theorem, the first summand converges as
\begin{align}
    \sqrt{\frac{\log \lambda_n}{ \frac{1}{2}\log n } } \left( \frac{ K_{0,\lambda_n} - \log \lambda_n }{ \sqrt{\log \lambda_n } } \right) \xrightarrow{D} \mathcal{N}(0,1),
\end{align}
as $n \to \infty$. Finally, the second summand can be rewritten as $\frac{ \log \lambda_n - \frac{1}{2}\log n }{ \sqrt{\frac{1}{2}\log n } } = \frac{\log\left( \frac{\lambda_n}{\sqrt{n}} \right)}{ \sqrt{\log(\sqrt{n})} }$, which converges in probability to $0$ as $n \to \infty$ by the same argument as before. Therefore another application of Slutsky's theorem yields
\begin{align}
\frac{ K_n - \frac{1}{2}\log n }{ \sqrt{\frac{1}{2}\log n } } \stackrel{D}{\longrightarrow} \mathcal{N}(0,1),
\end{align}
as $n \to \infty$. 
\end{proof}



Lastly, we illustrate that the family of random variables $\left\{ \frac{K_n(\pi_n)}{\log n}\right\}_{n \geq 1}$ satisfies a large deviation principle and compute the speed and rate function explicitly.

\begin{proof}[Proof of Theorem \ref{PFcyclesLDP}.]
For notational simplicity, let $K_n := K_n(\pi_n)$ and $\lambda_n := \lambda_n(\pi_n)$. 
We will use the G\"{a}rtner-Ellis theorem [Theorem 2.3.6, \cite{DZ}]. To this end, we begin by showing the existence of the limit
\begin{align}
    \lim_{n \to \infty} \frac{1}{\log n} \log\E\left[ \exp\left(tK_n\right) \right]
\end{align}
for all $t \in \R$. 

Let $K_{0,n}$ denote the number of cycles in a random permutation of size $n$. Recall from Equation (\ref{distrep}) that $K_n \stackrel{D}{=} K_{0,\lambda_n}$, where $\lambda_n$ is the number of cyclic points and $\stackrel{D}{=}$ denotes equality in distribution. Moreover, $K_{0,n}$ can be written as $K_{0,n} = \sum_{k=1}^n Y_k$, where $Y_k \sim \Ber(1/k)$ are independent Bernoulli random variables (see for example Chapter 3 of \cite{Pit06}).

The moment generating function of $K_n$ is computed as
\begin{align}
    \E[\exp\left(t K_n \right)] = \E[ \exp\left(t K_{0,\lambda_n} \right)] = \E\left[ \exp\left(t \sum_{k=1}^{\lambda_n} Y_k \right) \right],
\end{align}
where $Y_k \sim \Ber(1/k)$ are independent Bernoulli random variables. Using the independence of $\{Y_k\}$ and asymptotic properties of the Gamma function, 
\begin{align}
    \E\left[ \exp\left(t \sum_{k=1}^{\lambda_n} Y_k \right) \middle\vert \lambda_n\right] = \prod_{k=1}^{\lambda_n} \E\left[ \exp\left( tY_k \right) \right] = \prod_{k=1}^{\lambda_n} \frac{e^t + k - 1}{k} = \frac{1}{\Gamma(e^t)}\frac{\Gamma(e^t + \lambda_n)}{\Gamma(1 + \lambda_n)}. 
\end{align}
By the law of total expectation, 
\begin{align}
    \E[\exp\left(t K_n \right)] &= \E\left[\E\left[ \exp\left(t \sum_{k=1}^{\lambda_n} Y_k \right) \middle\vert \lambda_n\right]\right] =  \E\left[ \frac{\Gamma(1)}{\Gamma(e^t)}\frac{\Gamma(e^t + \lambda_n)}{\Gamma(1 + \lambda_n)} \right] \sim \frac{\E\left[ \lambda_n^{e^t - 1}\right]}{\Gamma(e^t)}
\end{align}
for $n$ large, since $\lambda_n \to \infty$ almost surely as $n \to \infty$. 
Therefore 
\begin{align}
\begin{split}
    \lim_{n \to \infty} \frac{1}{\log n} \log\E\left[ \exp\left(tK_n\right) \right] &= \lim_{n \to \infty} \frac{1}{\log n} \log\left( \frac{\E\left[ \lambda_n^{e^t - 1}\right]}{\Gamma(e^t)} \right) \\
    &= \lim_{n \to \infty} \frac{1}{\log n} \log\left(\frac{\E\left[ n^{\frac{1}{2}(e^t - 1)} \left(\frac{\lambda_n}{\sqrt{n}}\right)^{e^t - 1}\right] }{\Gamma(e^t)} \right) \\
    &= \lim_{n \to \infty} \frac{\frac{1}{2}(e^t - 1) \log n + \log\E\left[\left(\frac{\lambda_n}{\sqrt{n}}\right)^{e^t - 1}\right] - \log(\Gamma(e^t))}{\log n} \\
    &= \frac{1}{2}(e^t - 1) + \lim_{n \to \infty} \frac{\log\E\left[\left(\frac{\lambda_n}{\sqrt{n}}\right)^{e^t - 1}\right]}{\log n}. 
    \end{split}
\end{align}
Observe that for all $t \in \R$, we have that $e^t - 1 \in (-1, \infty)$. 
Thus by Lemma \ref{PFcyclicmomentconvergence}, 
\begin{align} \label{remainderlimitLDP}
\begin{split}
    &\lim_{n \to \infty} \frac{\log\E\left[\left(\frac{\lambda_n}{\sqrt{n}}\right)^{e^t - 1}\right]}{\log n} = \lim_{n \to \infty} \frac{\log\left(2^{(e^t - 1)/ 2} \Gamma\left( 1 + \frac{ e^t - 1}{2} \right)\right)}{\log n} = 0
    \end{split}
\end{align}
for fixed $t \in \R$ as $n \to \infty$.
Finally, note that $\{t : \frac{1}{2}(e^t - 1) < \infty\} = \R$ and that $t \mapsto \frac{1}{2}(e^t - 1)$ is differentiable. 

Therefore by the G\"{a}rtner-Ellis theorem, $\left\{\frac{K_n}{\log n}\right\}$ satisfies a large deviation principle with speed $\log n$ and rate function $I(x) = \sup_{t \in \R} \left\{tx - \frac{1}{2}(e^t - 1) \right\}$. The expression for $I(x)$ in the statement follows upon noting that the supremum is attained at $t = \log 2x$ if $x \geq 0$ and at $t = -\infty$ if $x < 0$.  
\end{proof}

\section{Lengths of the longest cycles}
\label{sec:len}
In this section, we study the lengths of the longest cycles in a random parking function. Building upon an observation of Kolchin \cite{Kol86} that the cyclic points of a random mapping form a random permutation, Finch \cite[Section 3]{Fin22} derived the joint density of the number of cyclic points $\lambda_n(f_n)$ and the length of the $r$th longest cycle $L_r(f_n)$ of a random mapping $f_n$ by multiplying the conditional probability of $L_r(f_n)$ given $\lambda_n(f_n)$ with the limiting density of $\lambda_n(f_n)$ and then differentiating with respect to $L_r(f_n)$. Recall from Remark \ref{rk:exp_cyclic} that the cyclic points of a random parking function form a random permutation, allowing us to adapt Finch's technique for random mappings to random parking functions. In what follows, let $L_r(\pi_n)$ be the length of the $r$th longest cycle of a uniformly random parking function $\pi_n \in \PF_n$, defined to be $0$ if the number of cycles $K_n(\pi_n)$ is less than $r$. We will obtain the joint distribution of the number of cyclic points $\lambda_n(\pi_n)$ and the length of the $r$th longest cycle $L_r(\pi_n)$ of a random parking function $\pi_n$. 

\begin{proposition} \label{CyclicLongestJointPDF}
    Let $L_r(\pi_n)$ be the length of the $r$th longest cycle of a uniformly random parking function $\pi_n \in \PF_n$. Let $\lambda_n(\pi_n)$ be the number of cyclic points of $\pi_n$. If $\frac{k}{\sqrt{n}} \to x > 0$ and $\frac{\ell}{\sqrt{n}} \to y > 0$ as $n \to \infty$, then
    \begin{align}
        \PR\left( \lambda_n(\pi_n) = k, L_r(\pi_n) \leq \ell \right) \sim \frac{1}{\sqrt{n}} xe^{-x^2/2} \rho_r\left( \frac{x}{y} \right), 
    \end{align}
    where $\rho_r$ is the $r$th generalized Dickman function.
\end{proposition}

\begin{proof}
We can rewrite the joint probability as
\begin{align}
\PR(\lambda_n(\pi_n) = k, L_r(\pi_n) \leq \ell) = \PR( L_r(\pi_n) \leq \ell \mid \lambda_n(\pi_n) = k) \PR(\lambda_n(\pi_n) = k). 
\end{align}
For $n \to \infty$, we have that $\PR(\lambda_n(\pi_n) = k) \sim \frac{1}{\sqrt{n}}xe^{-x^2/2}$ by Proposition \ref{PFcyclicPDF} and $\PR( L_r(\pi_n) \leq \ell \mid \lambda_n(\pi_n) = k) \sim \rho_r\left( \frac{x}{y} \right)$ by Proposition \ref{LongCyclesDickman}. The proof follows. 
\end{proof}

\begin{remark} \label{CyclicLongestAsympJointPDF}
Differentiating the asymptotic distribution function in Proposition \ref{CyclicLongestJointPDF} with respect to $y$ gives 
\begin{align}
\frac{d}{dy} xe^{-x^2/2} \rho_r\left( \frac{x}{y} \right) = \frac{x}{y} e^{-x^2/2} \left( -\frac{x}{y} \right) \rho_r'\left(\frac{x}{y} \right). 
\end{align} 
Using the delay differential equation (\ref{DDE}), we get that
\begin{align}
 -\frac{x}{y} \rho_r'\left(\frac{x}{y} \right) = \rho_r\left( \frac{x}{y} - 1 \right) - \rho_{r-1}\left( \frac{x}{y} - 1 \right), \qquad \text{for $x > y >0$ and $r\geq 1$.}
\end{align}
Thus $\left( \frac{\lambda_n(\pi_n)}{\sqrt{n}}, \frac{L_r(\pi_n)}{\sqrt{n}} \right)$ has a limiting joint density function given by 
\begin{align}
f(x,y) = \frac{x}{y} e^{-x^2/2} \left( \rho_r\left( \frac{x - y}{y}\right) - \rho_{r-1}\left( \frac{x - y}{y}\right) \right), \qquad \text{for $x > y >0$ and $r\geq 1$.}
\end{align}
\end{remark}

We now establish the asymptotic mean of $L_r(\pi_n)$ for any valid $r$. The proof is similar to the proof of the corresponding result for random mappings \cite{PW68}, although there the conditioning on the cyclic points was implicit. 

\begin{theorem} \label{LongCyclesAsymptoticMean}
    Let $L_r(\pi_n)$ be the length of the $r$th longest cycle of a uniformly random parking function $\pi_n \in \PF_n$. Then
    \begin{align}
        \frac{\E(L_r(\pi_n))}{\sqrt{n}} \to \sqrt{\frac{\pi}{2}} G_{r,1}
    \end{align}
    as $n \to \infty$, where $G_{r,1} := \frac{1}{\Gamma(r)} \int_0^\infty E_1(x)^{r-1} e^{-E_1(x) - x} \, dx$ with $E_1(x) = \int_x^\infty t^{-1} e^{-t} \, dt$.
\end{theorem}

\begin{proof}
Conditioning on the number of cyclic points and applying Proposition \ref{Shepp-LloydLimit}, we can write
\begin{align} \label{cyclelengthmeandecomp}
    \E\left[ \frac{L_r(\pi_n)}{\sqrt{n}} \right] &=\E\left[ \frac{1}{\sqrt{n}} \E\left[ L_r(\pi_n) \vert \lambda_n(\pi_n) \right] \right] = G_{r,1} \E\left[\frac{\lambda_n(\pi_n)}{\sqrt{n}} \right] + \E\left[\frac{\lambda_n(\pi_n)}{\sqrt{n}} \epsilon_{r,1,\lambda_n(\pi_n)} \right] \notag \\
    &= \E\left[\frac{\lambda_n(\pi_n)}{\sqrt{n}} \right]\left( G_{r,1} + \frac{\E[\lambda_n(\pi_n) \epsilon_{r,1,\lambda_n(\pi_n)}]}{\E[\lambda_n(\pi_n)]} \right).
\end{align}
By Lemma \ref{PFcyclicmean}, $\E\left[\frac{\lambda_n(\pi_n)}{\sqrt{n}} \right] \to \sqrt{\frac{\pi}{2}}$ as $n \to \infty$. 
Our claim follows once we show that 
\begin{align} \label{PWLimitTrick}
    \lim_{n \to \infty} \frac{\E\left[\lambda_n(\pi_n) \epsilon_{r,1,\lambda_n(\pi_n)} \right]}{\E[\lambda_n(\pi_n)]} = 0.
\end{align}

Fix $\delta > 0$. Since $|\epsilon_{r,1,k}| \to 0$ as $k \to \infty$, there exists $M > 0$ such that $|\epsilon_{r,1,k}| < \delta$ for all $k > M$. Thus for $n$ sufficiently large, we can write
\begin{align} \label{cyclesecondtermdecomp}
\begin{split}
    \left|\frac{\E\left[\lambda_n(\pi_n) \epsilon_{r,1,\lambda_n} \right]}{\E[\lambda_n(\pi_n)]} \right| &= \left| \frac{\sum_{k=1}^M kP(\lambda_n(\pi_n) = k) \epsilon_{r,1,k}}{\E[\lambda_n(\pi_n)]} + \frac{\sum_{k=M+1}^n kP(\lambda_n(\pi_n) = k) \epsilon_{r,1,k}}{\E[\lambda_n(\pi_n)]} \right| \\
    &\leq \frac{\sum_{k=1}^M kP(\lambda_n(\pi_n) = k) \left|\epsilon_{r,1,k}\right|}{|\E[\lambda_n(\pi_n)]|} + \delta.
\end{split}
\end{align}
On the other hand, taking $N$ sufficiently large so that 
\begin{align}
    \frac{\sum_{k=1}^M kP(\lambda_n(\pi_n) = k)}{\E[\lambda_n(\pi_n)]} < \delta
\end{align}
for all $n > N$, we have that
\begin{align}
    \frac{\sum_{k=1}^M kP(\lambda_n(\pi_n) = k) \epsilon_{r,1,k}}{\E[\lambda_n(\pi_n)]} < \delta \cdot \max_{1 \leq k \leq M} |\epsilon_{r,1,k}|
\end{align}
for all $n > N$. Since $\delta > 0$ was arbitrary, this proves (\ref{PWLimitTrick}) and the desired claim follows. 
\end{proof}

This result says that the average length of the longest cycle of a random parking function of size $n$ is asymptotically $0.7825\sqrt{n}$, the same as for random mappings. Recall from Section \ref{Sec: Prelims} that the average length of the longest cycle of a random permutation of size $n$ is asymptotically $0.6243n$.

\section{Limit theorems for prime parking functions}
\label{sec:prime}
In this section, we establish the corresponding limit theorems for random prime parking functions. The key result is the following lemma, which implies that the probability mass function of the number of cyclic points of a random prime parking function is asymptotically equal to that of a random classical parking function. 

\begin{lemma} \label{PPFcyclicPMF}
Let $\lambda_n(\pi_n)$ be the number of cyclic points of a uniformly random prime parking function $\pi_n \in \PPF_n$. Then
\begin{align}
\PR(\lambda_n(\pi_n) = k) = \frac{k (n-2)!}{(n-1)^k (n-1-k)!}
\end{align}
for all $1 \leq k \leq n-1$. 
\end{lemma}

\begin{proof}
Let $|\PPF_n^{(k)}|$ denote the number of prime parking functions of size $n$ with exactly $k$ cyclic points. Then $\PR(\lambda_n(\pi_n) = k) = \frac{|\PPF_n^{(k)}|}{(n-1)^{n-1}}$. Our claim readily follows from Theorem \ref{prime-cyclic}.
\end{proof}

From this result, we can see that the asymptotic behavior of $\lambda_n(\pi_n)$ in the prime parking functions case is the same as that in the classical parking functions case. In particular, for a random prime parking function we also have that 
\begin{align}
        \Lambda_n = \frac{\lambda_n(\pi_n)}{\sqrt{n}} \stackrel{D}{\longrightarrow} \Lambda \sim \Rayleigh(1)
\end{align}
as $n \to \infty$. Therefore all of our limit theorems for classical parking functions also hold for prime parking functions, with the same proofs. We summarize this in the following comprehensive theorem.

\begin{theorem} \label{PPFLimitTheorems}
Let $\pi_n \in \PPF_n$ be a uniformly random prime parking function of length $n$. 
\begin{enumerate}
\item Let $\lambda_n(\pi_n)$ be the number of cyclic points of $\pi_n$. If $k/\sqrt{n} \to x > 0$ as $n \to \infty$, then 
\begin{align}
    \PR(\lambda_n = k) \sim \frac{1}{\sqrt{n}}xe^{-x^2/2},
\end{align}
so that the scaled number of cyclic points $\Lambda_n = \frac{\lambda_n(\pi_n)}{\sqrt{n}} \xrightarrow{D} \Lambda$ as $n \to \infty$, where $\Lambda$ is a $\Rayleigh(1)$ random variable.
\item (Law of large numbers) Let $K_n(\pi_n)$ be the number of cycles of $\pi_n$. Then
\begin{align}
    \frac{K_n(\pi_n)}{\frac{1}{2}\log n} \xrightarrow{\PR} 1
\end{align}
as $n \to \infty$. 
\item (Central limit theorem) Let $K_n(\pi_n)$ be the number of cycles of $\pi_n$. Then
\begin{align}
\frac{K_n(\pi_n) - \frac{1}{2}\log n}{\sqrt{\frac{1}{2}\log n }} \xrightarrow{D} \mathcal{N}(0,1),
\end{align}
as $n \to \infty$, where $\mathcal{N}(0,1)$ is a standard normal random variable.
\item (Large deviation principle) The family $\left\{ \frac{K_n(\pi_n)}{\log n}\right\}_{n \geq 1}$ satisfies a large deviation principle with speed $\log n$ and rate function
\begin{align}
I(x) = \begin{cases}
x\log\left( 2x \right) - x + \frac{1}{2} & x \geq 0, \\
\infty & x < 0. 
\end{cases}
\end{align}
\item Let $L_r(\pi_n)$ be the length of the $r$th longest cycle of $\pi_n$, defined to be $0$ if the number of cycles, $K_n(\pi_n)$, is less than $r$. Then
\begin{align}
\frac{\E(L_r(\pi_n))}{\sqrt{n}} \to \sqrt{\frac{\pi}{2}} G_{r,1}
\end{align}
as $n \to \infty$, where $G_{r,1} := \frac{1}{\Gamma(r)} \int_0^\infty E_1(x)^{r-1} e^{-E_1(x) - x} \, dx$ with $E_1(x) = \int_x^\infty t^{-1} e^{-t} \, dt$.
\end{enumerate}
\end{theorem}

\section{Final remarks}

In this paper, we investigated the large $n$ asymptotic behavior of the cycle structure of random parking functions. In particular, we established the classical trio of limit theorems (law of large numbers, central limit theorem, large deviation principle) for the number of cycles $K_n(\pi_n)$ of a uniformly random parking function $\pi_n \in \PF_n$. This further implies that the equivalence of ensembles between parking functions and mappings holds for various statistics involving cycles. As stated in the introduction, understanding the cycle structure of random parking functions is a challenging research topic, as random parking functions do not satisfy nice exchangeability properties like random permutations and random mappings do.
Our study makes advances in this intriguing research direction, yet it leaves open several possible avenues for future research. 

\begin{itemize}
\item We have established a central limit theorem for the number of cycles of a random parking function. The next step is to refine this result by proving a functional central limit theorem. This was done for random mappings by Hansen \cite{Han89}.

\item We have left open the investigation of the asymptotic behavior of the tree components and the connected components of random parking functions. Much is known about the corresponding statistics on random mappings, for example, the asymptotic distributions of the number of trees of all sizes $1 \leq m \leq n$ and the size of the largest tree \cite{Ste69}, the asymptotic distribution of the number of cycle trees of size $r$ for all valid $r$ \cite{FO90}, and the expected value of the size of the largest tree and the size of the largest connected component \cite{FO90}. Although we expect the equivalence of ensembles to also hold for these characteristics, we believe that rigorously showing this will be quite difficult.

\item By conditioning on the number of cyclic points as in our proof techniques, we see that the process of normalized longest cycle lengths of a random parking function converges conditionally to the Poisson-Dirichlet distribution with parameter $1$, 
\begin{align}
    \left( \frac{L_1(\pi_n)}{\lambda_n(\pi_n)}, \frac{L_2(\pi_n)}{\lambda_n(\pi_n)}, \ldots \middle\vert \lambda_n(\pi_n) \right) \xrightarrow{D} \texttt{Poisson-Dirichlet}(1). 
\end{align}
Let $\tilde{L}_r(\pi_n)$ be the size of the $r$th largest component. We expect that the process of normalized largest component sizes, $\left( \frac{\tilde{L}_1(\pi_n)}{n}, \frac{\tilde{L}_2(\pi_n)}{n}, \ldots \right)$, of a random parking function also converges to the Poisson-Dirichlet distribution, but with parameter $\frac{1}{2}$, as shown by Aldous \cite{Ald85} in the random mappings case. Nevertheless, rigorously showing this will require an in-depth understanding of the component structure as pointed out in the last potential research direction.
\end{itemize}

\section*{Acknowledgements}

We thank David Aldous, Steven Finch, Svante Janson, Martin Rubey, and Richard Stanley for helpful discussions. The authors are grateful to many valuable comments from the referees, which significantly improved the quality of the paper. Mei Yin appreciated the opportunity to talk about this work at the 2025 International Conference on Probabilistic, Combinatorial and Asymptotic Methods for the Analysis of Algorithms.

\Address


\begin{thebibliography}{99}

\bibitem{Ald85}
D.~Aldous,
\textit{Exchangeability and related topics},
Lecture Notes in Mathematics 1117, Springer Verlag, New York, (1985).

\bibitem{AP94}
D.~Aldous and J.~Pitman,
\textit{Brownian bridge asymptotics for random mappings},
Random Structures and Algorithms, 5 (1994), 487-512.

\bibitem{AP02}
D.~Aldous and J.~Pitman,
\textit{The asymptotic distribution of the diameter of a random mapping},
Comptes Rendus Mathematique, 334 (11) (2002), 1021-1024.

\bibitem{AT92}
R.~Arratia and S.~Tavar\'{e},
\textit{The cycle structure of random permutations},
Annals of Probability, 20 (3) (1992), 1567-1591.

\bibitem{ABT}
R.~Arratia, A.~D.~Barbour, and S.~Tavar\'{e},
\textit{Logarithmic Combinatorial Structures: A Probabilistic Approach},
EMS Monographs in Mathematics, European Mathematical Society, (2003).

\bibitem{Bel23}
E.~Bellin,
\textit{Asymptotic behaviour of the first positions of uniform parking functions},
Journal of Applied Probability, 60 (2023), 1201-1218.


\bibitem{CM01}
P.~Chassaing and J.~F.~Marckert,
\textit{Parking functions, empirical processes, and the width of rooted labeled trees},
Electronic Journal of Combinatorics, 8 (2001), \#R14.

\bibitem{DP85}
J.~M.~DeLaurentis and B.~Pittel,
\textit{Random permutations and Brownian motion},
Pacific Journal of Mathematics, 119 (2) (1985), 287-301.

\bibitem{DZ}
A.~Dembo and O.~Zeitouni,
\textit{Large Deviations Techniques and Applications},
Springer Berlin, Heidelberg, (2010).

\bibitem{DH17}
P.~Diaconis and A.~Hicks,
\textit{Probabilizing parking functions},
Advances in Applied Mathematics, 89 (2017), 125-155.

\bibitem{Dic30}
K.~Dickman,
\textit{On the frequency of numbers containing prime factors of a certain relative magnitude},
Arkiv for Matematik, Astronomi och Fysik, 22 (10) (1930), 1-14.

\bibitem{DHHRY23}
I.~Durmi\'{c}, A.~Han, P.~E.~Harris, R.~Ribiero, and M.~Yin,
\textit{Probabilistic parking functions},
Electronic Journal of Combinatorics, 30 (3) (2023), \#P3.18.

\bibitem{FH98}
S.~Feng and F.~M.~Hoppe,
\textit{Large deviation principles for some random combinatorial structures in population genetics and Brownian motion},
Annals of Applied Probability, 8 (4) (1998), 975-994.

\bibitem{Fin22}
S.~Finch,
\textit{Components and cycles of random mappings},
arXiv:2205.05579 (2022).

\bibitem{FO90}
P.~Flajolet and A.~Odlyzko,
\textit{Random mapping statistics},
In: Advances in Cryptology - EUROCRYPT '89, Lecture Notes in Computer Science, 434 (1990), 329-354.

\bibitem{FPV98}
P.~Flajolet, P.~Poblete, and A.~Viola,
\textit{On the analysis of linear probing hashing},
Algorithmica, 22 (1998), 490-515.

\bibitem{FR74}
D.~Foata and J.~Riordan,
\textit{Mappings of acyclic and parking functions},
Aequationes Math, 10 (1) (1974), 10-22.

\bibitem{Gol64}
S.~W.~Golomb,
\textit{Random permutations},
Bulletin of the American Mathematical Society, 70 (1964), 747.

\bibitem{Gon44}
V.~Goncharov,
\textit{Some facts from combinatorics},
Izvestia Akad. Naukl. SSSR, Ser. Mat., 8 (1944), 3-48.

\bibitem{Han89}
J.~Hansen,
\textit{A functional central limit theorem for random mappings},
Annals of Probability, 17 (1) (1989), 317-332.

\bibitem{Har60}
B.~Harris,
\textit{Probability distributions related to random mappings},
Annals of Mathematical Statistics, 31 (4) (1960), 1045-1062.

\bibitem{Har25}
P.~E.~Harris, T.~Holleben, J.~C.~Mart\'{i}nez Mori, A.~Priestley, K.~Sullivan, and P.~Wagenius,
\textit{A probabilistic parking process and labeled IDLA},
arXiv:2501.11718 (2025).

\bibitem{Hwa96}
H.~K.~Hwang,
\textit{Large deviations for combinatorial distributions. I: Central limit theorems},
Annals of Applied Probability, 6 (1) (1996), 297-319.

\bibitem{Jan01}
S.~Janson,
\textit{Asymptotic distribution for the cost of linear probing hashing},
Random Structures \& Algorithms, 19 (3-4) (2001), 438-471.

\bibitem{Kal99}
L.~Kalikow,
\textit{Enumeration of parking functions, allowable permutation pairs, and labeled trees},
PhD Thesis, Brandeis University, (1999).

\bibitem{KY21}
R.~Kenyon and M.~Yin,
\textit{Parking functions: From combinatorics to probability},
Methodology and Computing in Applied Probability, 25 (32) (2023).

\bibitem{KY18}
W.~King and C.~H.~Yan,
\textit{Parking functions on oriented trees},
S\'{e}minaire Lotharingien de Combinatoire, 80B (47) (2018).

\bibitem{KT76}
D.~E.~Knuth and L.~Trabb Pardo,
\textit{Analysis of a simple factorization algorithm},
Theoretical Computer Science, 3 (1976), 321-348;
also in \textit{Selected Papers on Analysis of
Algorithms}, CSLI, (2000), 303-339.

\bibitem{Kol76}
V.~F.~Kolchin,
\textit{A problem of the allocation of particles in cells and random mappings},
Theory of Probability and its Applications, 21 (1) (1976), 48-63.

\bibitem{Kol86}
V.~F.~Kolchin,
\textit{Random Mappings},
Optimization Software Incorporation, (1986), 35, 46-48, 85-88, 93-94, 153.

\bibitem{KW66}
A.~G.~Konheim and B.~Weiss,
\textit{An occupancy discipline and applications},
SIAM Journal on Applied Mathematics, 14 (1966), 1266-1274.

\bibitem{Kin77}
J.~F.~C.~Kingman,
\textit{The population structure associated with the Ewens sampling formula},
Theoretical Population Biology, 11 (1977), 274-283.

\bibitem{LP16}
M.-L.~Lackner and A.~Panholzer,
\textit{Parking functions for mappings},
Journal of Combinatorial Theory, Series A, 142 (2016), 1-28.

\bibitem{Mac95}
I.~G.~Macdonald,
\textit{Symmetric functions and Hall polynomials},
Oxford Mathematical Monographs, The Clarendon Press, Oxford University Press, Second Edition, (1995).

\bibitem{MF24}
L.~Mutafchiev and S.~Finch,
\textit{On the deepest cycle of a random mapping},
Journal of Combinatorial Theory, Series A, 206 (2024), 105875.

\bibitem{Pag23}
J.~E.~Paguyo,
\textit{Cycle structure of random parking functions},
Advances in Applied Mathematics, 144 (2023), 102458.

\bibitem{Pit01}
J.~Pitman,
\textit{Random mappings, forests, and subsets associated with Abel-Cayley-Hurwitz multinomial expansions},
S\'{e}minaire Lotharingien de Combinatoire, 46 (2001), Article B46h.

\bibitem{Pit06}
J.~Pitman,
\textit{Combinatorial Stochastic Processes},
Ecole d'Et\'{e} de Probabilit\'{e}s de Saint-Flour XXXII. Lecture Notes in Mathematics, 1875 (2006), Springer, Berlin.

\bibitem{PW68}
P.~W.~Purdom and J.~H.~Williams,
\textit{Cycle length in a random function},
Transactions of the American Mathematical Society, 133 (2) (1968), 547-551.

\bibitem{Pollak2}
J.~Riordan,
\textit{Ballots and trees},
Journal of Combinatorial Theory, 6 (1969), 408-411.

\bibitem{Riordan}
J.~Riordan,
\textit{An Introduction to Combinatorial Analysis},
Princeton University Press, Princeton, (1978).

\bibitem{RY25}
M.~Rubey and M.~Yin,
\textit{Fixed points and cycles of parking functions},
Enumerative Combinatorics and Applications, 5:2 (2025), Article S2R10.

\bibitem{SL66}
L.~A.~Shepp and S.~P.~Lloyd,
\textit{Ordered cycle lengths in a random permutation},
Transactions of the American Mathematical Society, 121 (1966), 340-357.

\bibitem{SY23}
R.~P.~Stanley and M.~Yin,
\textit{Some enumerative properties of parking functions},
arXiv:2306.08681 (2023).

\bibitem{Ste69}
V.~E.~Stepanov,
\textit{Limit distributions of certain characteristics of random mappings},
Theory of Probability and its Applications, 14 (4) (1969), 612-626.

\bibitem{Yan15}
C.~H.~Yan,
\textit{Parking functions},
In: Handbook of Enumerative Combinatorics, Discrete Math. Appl. CRC Press, Boca Raton, (2015), 835-893. 

\bibitem{Yin23I}
M.~Yin,
\textit{Parking functions: Interdisciplinary connections},
Advances in Applied Probability, 55 (2023), 768-792.

\bibitem{Yin23II}
M.~Yin,
\textit{Parking functions, multi-shuffle, and asymptotic phenomena},
La Matematica, 2 (2023), 258-282.
  
\end{thebibliography}
\end{document}